\documentclass[11pt,a4paper,draft]{article}

\usepackage{amsmath,amsthm,amsopn,amsfonts,amssymb,dsfont,color}
\usepackage{fancybox}
\usepackage{pst-all}
\usepackage{comment}
\usepackage{relsize}
\usepackage{pstricks}
\usepackage{pstricks-add}
\usepackage{pst-plot}
\usepackage{pst-func}
\usepackage{graphicx,graphics}

\usepackage{a4,a4wide}

\def\ep{{\varepsilon}}
\def\R{\mathbb R}
\def\I{\mathcal I}

\newtheorem{theo}{\textbf{Theorem}}[section]

\newtheorem{lem}[theo]{\textbf{Lemma}}
\newtheorem{prop}[theo]{\textbf{Proposition}}
\newtheorem{cla}[theo]{\textbf{Claim}}
\newtheorem{cor}[theo]{\textbf{Corollary}}
\newtheorem{defi}[theo]{\textbf{Definition}}

\newtheorem{assumption}[theo]{\textbf{Assumption}}
\newtheorem{rem}[theo]{\textbf{Remark}}

\definecolor{cya}{RGB}   {0,   174,   240 }

\title{Propagation phenomena in monostable integro-differential equations:  acceleration or not?}

\date{}

\begin{document}

\maketitle
\begin{center}
{\large\bf Matthieu Alfaro \footnote{ IMAG, Universit\'e de
Montpellier, CC051, Place Eug\`ene Bataillon, 34095 Montpellier
Cedex 5, France. E-mail: matthieu.alfaro@umontpellier.fr} and
J{\'e}r{\^o}me Coville \footnote{Equipe BIOSP, INRA Avignon,
Domaine Saint Paul, Site Agroparc, 84914 Avignon Cedex 9, France.
E-mail: jerome.coville@inra.fr}.}\\
[2ex]
\end{center}



\tableofcontents

\vspace{10pt}

\begin{abstract}
We consider the homogeneous integro-differential equation
$\partial _t u=J*u-u+f(u)$ with a monostable nonlinearity $f$. Our interest is twofold: we investigate the existence/non existence of travelling waves, and  the propagation properties of the Cauchy problem.

When the dispersion kernel $J$ is exponentially bounded, travelling waves are known to exist and solutions of the Cauchy problem typically propagate at a constant speed \cite{Schumacher1980}, \cite{Weinberger1982}, \cite{Carr2004}, \cite{Coville2007a}, \cite{Coville2008a}, \cite{Yagisita2009}. 
On the other hand, when the dispersion kernel $J$ has heavy tails and  the  nonlinearity $f$ is non degenerate, i.e $f'(0)>0$,  travelling waves do not exist and  solutions of the Cauchy problem  propagate by accelerating \cite{Medlock2003}, \cite{Yagisita2009}, \cite{Garnier2011}.  For a general monostable nonlinearity, a dichotomy between these two types of propagation behavior is still not known. 

 The originality of our work is to provide such dichotomy by  studying the interplay  between the tails of the dispersion kernel and the Allee  effect induced by the degeneracy of $f$, i.e. $f'(0)=0$. First, for algebraic decaying kernels, we prove the exact separation between existence and non existence of travelling waves. This in turn provides the exact separation between non acceleration and acceleration in the Cauchy problem.  In the latter case, we provide a first estimate of the position of the level sets of the solution.

\end{abstract}

\maketitle

\section{Introduction} \label{s:intro}

In this work, we are interested in the propagation phenomena for solutions $u(t,x)$ of homogeneous  integro-differential equations of the form
\begin{equation}\label{eq}
\partial _t u=J* u-u+f(u), \quad t>0,\, x\in \R.
\end{equation}
In population dynamics models, $u$ stands for a population density, the nonlinearity $f$ encodes the demographic assumptions and $J$ is a nonnegative dispersal  kernel of total mass 1, allowing to take into account long distance dispersal events. Here, we consider  nonlinearities $f$  of the monostable type, namely $f(0)=f(1)=0$ and $f>0$ on $(0,1)$. Precise assumptions on $J$ (heavy tails) and $f$ (degeneracy at 0) will be given later one.
\medskip

When $f$ is non degenerate at $0$, that is $f'(0)>0$, it is known that  the equation \eqref{eq} exhibits some propagation phenomena: starting with some nonnegative nontrivial compactly supported initial data, the corresponding solution $u(t,x)$ converges to $1$, its stable steady state,  at large time and locally uniformly in space. This is referred as the {\it hair trigger effect} \cite{Aronson1978}. Moreover, in many cases, the convergence to 1 can be precisely characterised. For example, when $f$ is a KPP nonlinearity ---meaning $f(s)\leq f'(0)s$ for all $s\in(0,1)$--- and $J$ is exponentially bounded, that is
\begin{equation} \label{mollison}
\exists \lambda>0, \quad \int_{\R}J(z)e^{\lambda |z|}\,dz <+\infty,
\end{equation}
equation \eqref{eq} admits travelling waves whose minimal speed $c^*$   completely characterises the convergence $u(t,x) \to 1$, see \cite{Schumacher1980}, \cite{Weinberger1982}, \cite{Carr2004}, \cite{Coville2007a}, \cite{Coville2008a}.

For non degenerate monostable nonlinearities $f$, when the condition  \eqref{mollison} is relaxed,  allowing dispersion kernels with {\it heavy tails},  a new propagation phenomena appears: {\it acceleration}. This phenomenon
for equation \eqref{eq} was first heuristically obtained by Medlock and Kot \cite{Medlock2003} and mathematically described in \cite{Yagisita2009},
\cite{Garnier2011}:  Yagisita \cite{Yagisita2009} proves the non existence of travelling waves, and Garnier \cite{Garnier2011}  studies the acceleration  in the Cauchy problem. 

\begin{rem}
Acceleration phenomena for positive solutions of a Cauchy problem also appear  in other contexts ranging from standard reaction diffusion equations \cite{Hamel2010b}, \cite{Alfaro2015a}, to homogeneous equations involving fractional operators \cite{Cabre2013}. Let us also mention that acceleration phenomenon also appears in some porous media equations \cite{King2003}, \cite{Stan2014}. 
\end{rem}

To capture this acceleration phenomenon, a precise description  of the behavior of the level sets of $u(t,x)$ is  required. More precisely,  for $\lambda\in (0,1)$, let $E_\lambda(t)$ denote the  set 
$$
E_\lambda(t):=\{x\in \R: u(t,x)=\lambda\}.
$$
Then the acceleration  can be characterised through the properties of  $x^{\pm}_\lambda(t)$ representing the ``largest'' and the ``smallest'' element of $E_{\lambda}(t)$, i.e $x^+_{\lambda}(t)=\sup E_\lambda(t)$ and $x^-_{\lambda}(t)=\inf E_\lambda(t) $. 

For example, when $f$ is a KPP  nonlinearity and $J(z)\sim \frac{C}{|z|^{\alpha}}$ ($\alpha>2$) for large $z$, 
the results of Garnier \cite{Garnier2011} assert that, for a solution of the Cauchy problem \eqref{eq} with a nonnegative compactly supported initial data, the points $x_\lambda^{\pm}(t)$   move exponentially fast at large time:  for any $\lambda \in (0,1)$ and $\ep>0$, there exists $\rho>f'(0)$ and $T_{\lambda,\ep}>0$ such that 
$$
e^{\frac{(f'(0)-\ep)}{\alpha}t} \le |x_{\lambda}^{\pm}(t)|\le e^{\frac{\rho}{\alpha} t},\; \forall t\ge T_{\lambda,\ep}.
$$
Similarly, Garnier \cite{Garnier2011} shows that, when $f$ is a KPP nonlinearity and $J(z)\sim Ce^{-\beta |z|^{\alpha}}$ ($0<\alpha<1$, $\beta >0$) for large $z$, the points $x_\lambda^{\pm}(t)$  move algebraically fast at large time:  for any $\lambda \in (0,1)$ and $\ep>0$, there exists $\rho>f'(0)$ and $T_{\lambda,\ep}>0$ such that 
$$
\left(\frac{f'(0)-\ep}{\beta}\right)^{\frac 1 \alpha } t^{\frac{1}{\alpha}} \le |x_{\lambda}^{\pm}(t)|\le \left(\frac \rho \beta \right)^{\frac 1 \alpha} t^{\frac{1}{\alpha}},\; \forall t\ge T_{\lambda,\ep}.
$$
Note that the lower and upper bounds in the above estimates do not agree: tracking the level sets when acceleration occurs is a quite challenging task.

\medskip
     
From a modelling point of view, it is natural to consider  equation \eqref{eq} with a monostable nonlinearity $f$ which degenerates at 0, that is $f'(0)=0$. This corresponds to assuming that the growth of the population at low density is not exponential any longer. In particular, this assumption induces an  Allee effect  on the evolution of the population, that is the maximal rate of production of new individuals is not achieved at low density, \cite{Allee1938}, \cite{Dennis1989}, \cite{Veit1996}, \cite{Stephens1999}, \cite{Kramer2009}.

In this $f'(0)=0$ degenerate case, the characterisation of the existence of travelling waves or acceleration  in terms of the tails of $J$ is far from understood. Indeed, under condition \eqref{mollison} for the kernel $J$, travelling fronts are known to exist \cite{Coville2007a}, \cite{Coville2008a}, and the Cauchy problem typically does not lead to acceleration  \cite{Zhang2012}. But, when condition \eqref{mollison} is relaxed, the competition between heavy tails and the Allee effect has not yet been studied.

The purpose of this work is  therefore to fill the above gap in the comprehension of propagation phenomena for equation \eqref{eq},  
thus completing the  picture describing the  dichotomy between the existence of  an accelerated propagation or not. First,  in the spirit of  \cite{Gui2015} for a fractional version of \eqref{eq}, we derive an exact relation between the algebraic tails of the kernel $J$ and the behavior of $f$ near zero, which allows or not
the existence of travelling waves.  Then,  in the spirit of \cite{Alfaro2015a} for a reaction diffusion equation with Allee effect and initial datum having heavy tails, we investigate the propagation phenomenon occurring in the Cauchy problem \eqref{eq} with front like initial data. As a consequence of the travelling waves analysis, we derive the exact separation between non acceleration and acceleration. In the latter case, we give some estimates on the \lq\lq speed'' of expansion of  the  level-sets of the solution.

\section{Assumptions and main results}

Before stating our results, let us first present our assumptions on the dispersal kernel $J$ and the degenerate monostable nonlinearity $f$. 

\begin{assumption}[Dispersal  kernel for existence of waves]\label{ass:J} $J:\R :\to [0,\infty)$ is continuous, of total mass
$ \int _\R J(x)dx=1 $. We assume that there is $C>0$ such that
\begin{equation}\label{hyp-queue-gauche}
J(x)\leq \frac{C}{\vert x\vert^{\mu}}, \quad \forall x\leq -1,
\;\text{ for some } \mu >2,
\end{equation}
and
\begin{equation}\label{hyp-queue-droite}
J(x)\leq \frac{C}{x^{\alpha}}, \quad \forall x\geq 1,\; \text{ for
some } \alpha >2.
\end{equation}
\end{assumption}

Symmetry is not assumed. As will be clear in the following, the important tail is the right one. In order to prove non existence we need to assume slightly more.

\begin{assumption}[Dispersal kernel for non existence of waves]\label{ass:J-noTW} $J$ satisfies Assumption \ref{ass:J} with \eqref{hyp-queue-droite} replaced by 
\begin{equation}\label{hyp-queue-droite-noTW}
\frac{1/C}{x^{\alpha}}
\leq J(x)\leq \frac{C}{x^{\alpha}}, \quad \forall x\geq 1,\; \text{ for
some } \alpha >2.
\end{equation}
\end{assumption}

\begin{assumption}[Degenerate monostable nonlinearity]\label{ass:f} $f:\left[0,1 \right]\to [ 0,\Vert f \Vert _\infty ]$ is of the class $C^{1}$, and is of the monostable type, in the sense that
$$
f(0)=f(1)=0, \quad f>0 \quad \text{ on } (0,1).
$$
The steady state $0$ is degenerate, in the sense that
\begin{equation}\label{hyp-f-zero}
f(u)\sim r  u^{\beta}, \text{ as } u\to 0, \; \text{ for some } r>0, \beta >1,
\end{equation}
whereas the steady state 1 is stable, in the sense that
\begin{equation}\label{hyp-f-un}
f'(1)<0.
\end{equation}
\end{assumption}

The simplest example of a monostable nonlinearity involving such a degenerate Allee effect  is given by $f(u)=r u^\beta(1-u)$.

\begin{defi}[Travelling wave] A travelling wave for equation \eqref{eq} is a couple $(c,u)$ where $c\in\R$ is the speed, and $u$ is a decreasing  profile satisfying
$$
\begin{cases}
J*u-u+cu'+f(u)=0 \quad \text{ on } \R,\\
u(-\infty)=1, \quad u(\infty)=0.
\end{cases}
$$
\end{defi}

Notice that if $c\neq 0$ then it follows from the equation that the profile $u$ of a travelling wave has to be in $C^1_b(\R)$. On the other hand, if $c=0$, the situation is more tricky and, as observed in \cite{Coville2008a}, it may happen that the above travelling wave problem admits infinitely many solutions that are not continuous. 

\begin{theo}[Existence of travelling waves]\label{th:tw} Let Assumptions \ref{ass:J} and \ref{ass:f} hold. Assume
\begin{equation}\label{alpha-beta}
\beta \geq 1+ \frac 1{\alpha -2}.
\end{equation}
Then there is $c^{*}>J_1:=\int _\R yJ(y)dy$ such that for all $c\geq c^{*}$ equation \eqref{eq} admits travelling waves $(c,u)$, whereas, for all $c<c^{*}$ equation \eqref{eq} does not admit travelling wave.
\end{theo}

On the one hand, for any (\lq\lq small'') $\beta >1$ (measuring the
degeneracy of $f$ in 0), one can find some (large) $\alpha$
(measuring the right tail of the kernel $J$) so that \eqref{eq}
supports the existence of travelling waves. On the other hand, for
any (\lq\lq small'') $\alpha
>2$, one can find some (large) $\beta$ so that \eqref{eq} supports the existence of travelling waves.

\begin{cor}[Kernels lighter than algebraic]\label{cor:expracine} Let Assumption \ref{ass:J} hold, with \eqref{hyp-queue-droite} replaced by: for all $\alpha >2$, there is $C_\alpha >0$ such that
\begin{equation}\label{hyp-queues-exp-lente}
J(x)\leq \frac{C_\alpha}{\vert x\vert^{\alpha}}, \quad \forall
x\geq 1.
\end{equation}
Let Assumption \ref{ass:f} hold. Then there is $c^{*}$ such that for all $c\geq c^{*}$ equation \eqref{eq} admits travelling waves $(c,u)$, whereas, for all $c<c^{*}$ equation \eqref{eq} does not admit travelling wave.
\end{cor}

The above result is independent on $\beta >1$ and is valid, among others, for kernels satisfying
$$
J(x)\leq C e^{-a\vert x\vert /(\ln \vert x \vert)}, \quad \forall
x\geq 2, \quad \text{ for some } a>0,
$$
or
$$
J(x)\leq C e^{-a\vert x\vert^{b}}, \quad \forall x\geq 1, \quad
\text{ for some } a>0,\, 0<b<1,
$$
for which travelling waves do not exist in the KPP case
\cite{Garnier2011}. The proof is obvious: for a given $\beta
>1$, select a large $\alpha>2$ such that \eqref{alpha-beta} holds,
and then combine \eqref{hyp-queues-exp-lente} with Theorem
\ref{th:tw} (more precisely the fact that the construction of an adequate supersolution is enough to prove the theorem, see Section \ref{s:tw}).

Similarly, for strongly degenerate monostable $f$, we have the following consequence.
 
\begin{cor}[Strongly degenerate nonlinearity]\label{cor:fzeldo} Let Assumption \ref{ass:J} hold. Let Assumption \ref{ass:f} hold, with \eqref{hyp-f-zero} replaced by: for all $\beta >1$, there is $C_\beta >0$ such that
\begin{equation}\label{hyp-f-zeldo}
\lim_{u\to 0^+}\frac{f(u)}{u^\beta}\leq C_\beta.
\end{equation}
 Then there is $c^{*}$ such that for all $c\geq c^{*}$ equation \eqref{eq} admits travelling waves $(c,u)$, whereas, for all $c<c^{*}$ equation \eqref{eq} does not admit travelling wave.
\end{cor}

The above result is independent on $\alpha >2$ and is valid, among others, for the Zel'dovich nonlinearity \cite{Zeldovich1948}, \cite{Kanel1961},  that is
$$
f(u)=  re^{-\frac{1}{u}}(1-u), \quad \text{ for some } r>0.
$$
The proof is again obvious: for a given $\alpha>2$, select a large $\beta>1$ such that \eqref{alpha-beta} holds,
and then combine \eqref{hyp-f-zeldo} with Theorem
\ref{th:tw} to construct an adequate supersolution.

Next, we prove that the hyperbola separation \eqref{alpha-beta} arising in Theorem \ref{th:tw} is optimal for the  existence of travelling wave.

\begin{theo}[Non existence of travelling wave]\label{th:notw} Let Assumptions \ref{ass:J-noTW} and \ref{ass:f} hold. Assume
\begin{equation}\label{alpha-beta-notw}
\beta <1+\frac{1}{\alpha -2}.
\end{equation}
Then there is no travelling wave $(c,u)$ for  equation \eqref{eq}.
\end{theo}

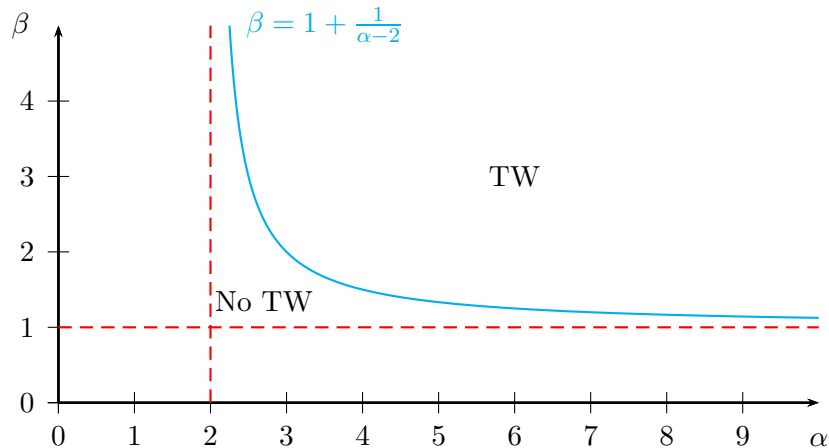
\psset{xunit=1cm,yunit=1cm,algebraic=true}
\def\xmin{0} \def\xmax{10}
\def\ymin{0} \def\ymax{5}
\begin{center}
\begin{pspicture}(\xmin,\ymin)(\xmax,\ymax)
\psaxes[showorigin=true, axesstyle=axes]{->}(0,0)(\xmax,\ymax)
\rput(10,-0.5){\psframebox*[]{$\alpha$}}
\rput(-0.5,5){\psframebox*[]{$\beta$}}
\rput(3.5,5){\psframebox*[]{\textcolor{cya}{$\beta=1+\frac{1}{\alpha-2}$}}}
\rput(2.7,1.35){\psframebox*[]{No TW}}
\rput(6,3){\psframebox*[]{TW}}
\def\f{1+ (1/(x-2))} 
\psplot[plotpoints=1000,linecolor=cya]{2.25}{\xmax}{\f} 
\psline[linestyle=dashed, linecolor=red](\xmin,1)(\xmax,1) 
\psline[linestyle=dashed, linecolor=red](2,\ymin)(2,\ymax) 
\end{pspicture}
\end{center}

\bigskip
We now turn to the Cauchy problem \eqref{eq} with a front like initial datum $u_0$.
   
\begin{assumption}[Front like initial datum]\label{ass:u0} $u_0$ is of the class $C^{1}$, and satisfies 
\begin{itemize}
\item[(i)] $0\le u_0(x)<1, \quad \forall\, x \in \R,$
\item[(ii)] $ \liminf_{x\to -\infty}u_0 (x)>0$,
\item[(iii)] $u\equiv 0$ on $[a,\infty)$ for some $a\in \R$.
\end{itemize}
\end{assumption}

Since $f$ is Lipschitz and $0\le u_0 \le 1$ the existence of a unique local solution $u(t,x)$ to the Cauchy problem \eqref{eq} with initial datum $u_0$  is rather classical. Moreover, from the strong maximum priniciple, we know that $0<u(t,\cdot)<1$ as soon as $t>0$ and the solution is global in time.

Theorem \ref{th:tw} and Theorem \ref{th:notw}, are strong indications that, under assumption \eqref{alpha-beta},  assumption \eqref{alpha-beta-notw}, no acceleration, respectively acceleration, should occur for the solution of the Cauchy problem. In order to clearly state such a result, for any $\lambda \in(0,1)$ we define, in the spirit of the level sets used in \cite{Hamel2010b}, \cite{Garnier2011}, \cite{Alfaro2015a}, the (super) level sets of a solution $u(t,x)$ by 
$$
\Gamma_{\lambda}(t):=\{x\in \R: u(t,x)\ge\lambda\}.
$$
Also we define the \lq\lq largest'' element of $\Gamma _\lambda (t)$ by
$$
x_\lambda (t):=\sup \Gamma_\lambda(t) \in \R\cup\{-\infty,\infty\}.
$$
Notice that, for compactly supported initial datum, it may happen that the solution get extinct at large time, which is referred as the {\it quenching phenomenon} \cite{Alfaro2016}, and thus $\Gamma _\lambda (t)=\emptyset$ at large time. This is one of our motivations for considering a front like initial datum.
We can now state our first result on the Cauchy problem.

\begin{prop}[Acceleration or not in the Cauchy problem]\label{prop:level-set}
Let Assumptions \ref{ass:J-noTW} and \ref{ass:f} hold. Let $u(t,x)$ be the solution of the Cauchy problem \eqref{eq} with an initial datum $u_0$ satisfying Assumption \ref{ass:u0}. 
\begin{itemize}
\item[(i)] Assume $\beta \ge 1+\frac{1}{\alpha -2}$.
Then there is $c_0\in \R$ such that, for any $\lambda \in (0,1)$, 
\begin{equation}
\label{conclusion}
\limsup _{t\to\infty}\frac{x_{\lambda}(t)}{t}\leq c_0.
\end{equation}
\item[(ii)]
Assume $\beta <1+\frac{1}{\alpha -2}$.
Then, for any $A \in \R$,
\begin{equation}\label{invasion}
\lim_{t\to \infty}u(t,x)=1 \quad \text{ uniformly in  } (-\infty,A],
\end{equation}
and, for any $\lambda \in (0,1)$,  
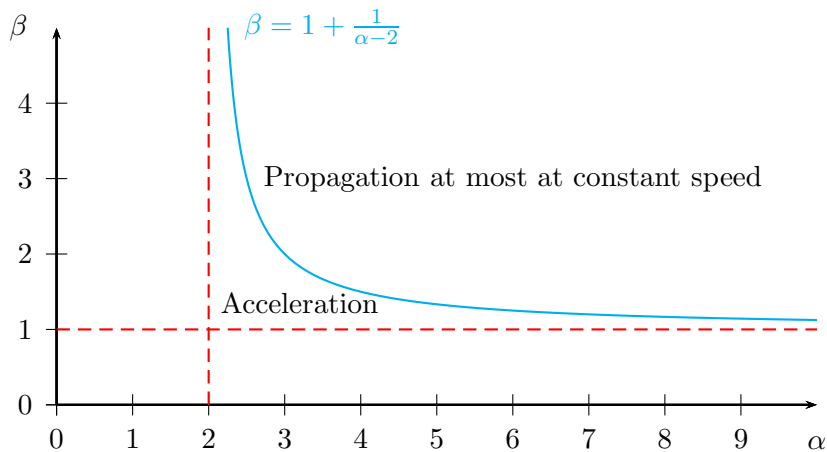
\begin{equation}\label{acc-prop}
\lim_{t\to\infty}\frac{x_{\lambda}(t)}{t}=+\infty.
\end{equation} 
\end{itemize}
\end{prop}
\psset{xunit=1cm,yunit=1cm,algebraic=true}
\def\xmin{0} \def\xmax{10}
\def\ymin{0} \def\ymax{5}
\begin{center}
\scalebox{1}{\begin{pspicture}(\xmin,\ymin)(\xmax,\ymax)
\psaxes[showorigin=true, axesstyle=axes]{->}(0,0)(\xmax,\ymax)
\rput(10,-0.5){\psframebox*[]{$\alpha$}}
\rput(-0.5,5){\psframebox*[]{$\beta$}}
\rput(3.5,5){\psframebox*[]{\textcolor{cya}{$\beta=1+\frac{1}{\alpha-2}$}}}
\rput(3.2,1.35){\psframebox*[]{Acceleration}}
\rput(6,3){\psframebox*[]{Propagation at most at constant speed}}
\def\f{1+ (1/(x-2))} 
\psplot[plotpoints=1000,linecolor=cya]{2.25}{\xmax}{\f} 
\psline[linestyle=dashed, linecolor=red](\xmin,1)(\xmax,1) 
\psline[linestyle=dashed, linecolor=red](2,\ymin)(2,\ymax) 
\end{pspicture}
}
\end{center}

\medskip

In the first situation $(i)$, we recover that the level sets of the solution $u(t,x)$ move to the right at most at a constant speed. Notice that the proof of $(i)$ is rather standard (use the supersolution of Theorem \ref{th:supersol} to control the propagation in the parabolic problem as in \cite[Section 3]{Alfaro2015a}) and will be omitted. Notice also that assuming further that 
$\limsup _{x\to -\infty}u_0(x)<1$, we can use the travelling wave with minimal speed as a supersolution, and thus replace $c_0$ by $c^*$ in the conclusion \eqref{conclusion}. But, due to the lack of symmetry of the kernel  $J$, it may happen that $c^*\le 0$, see \cite{Coville2008a}. In such a case, we observe a propagation failure phenomenon for the solutions of the Cauchy problem.   

On the other hand, the first part \eqref{invasion} of $(ii)$ shows that invasion does occur (in particular $\Gamma _\lambda (t)\neq \emptyset$ at large time). Moreover the second pat \eqref{acc-prop} of $(ii)$ indicates that the level sets of the solution move by accelerating.

Our last main result aim at precising the acceleration phenomenon $(ii)$, by giving a first estimate of the actual position of $x_\lambda (t)$.

\begin{theo}[Further estimates on the acceleration phenomenon] \label{th::speedacc}
Let Assumptions \ref{ass:J-noTW} and \ref{ass:f} hold. Let $u(t,x)$ be the solution of the Cauchy problem \eqref{eq} with an initial datum $u_0$ satisfying Assumption \ref{ass:u0}. 
Assume that  
\begin{equation}\label{alpha-beta2}
\beta <1+\frac{1}{\alpha -2}.
\end{equation}
Then there exists $\overline C>0$ such that for all $\lambda\in (0,1)$, there is $T_\lambda >0$ such that 
\begin{equation}\label{upper}
x_{\lambda}(t)\le \overline Ct^{\frac{1}{(\alpha-1)(\beta-1)}+\frac{1}{\alpha -1}}, \quad \forall t\geq T_\lambda.
\end{equation}
Moreover, under the stronger assumption $\beta <1+\frac{1}{\alpha -1}$, there exists $\underline C>0$ such that for all $\lambda\in (0,1)$, there is $ T_\lambda ' >0$ such that
\begin{equation}\label{lower}
\underline Ct^{\frac{1}{(\alpha-1)(\beta-1)}}\le x_{\lambda}(t)\le \overline Ct^{\frac{1}{(\alpha-1)(\beta-1)}+\frac{1}{\alpha -1}},\quad \forall t\geq  T_\lambda'.
\end{equation}
\end{theo}  
\psset{xunit=1cm,yunit=1cm,algebraic=true}
\def\xmin{0} \def\xmax{10}
\def\ymin{0} \def\ymax{5}
\begin{center}
\scalebox{1}{\begin{pspicture}(\xmin,\ymin)(\xmax,\ymax)
\psaxes[showorigin=true, axesstyle=axes]{->}(0,0)(\xmax,\ymax)
\rput(10,-0.5){\psframebox*[]{$\alpha$}}
\rput(-0.5,5){\psframebox*[]{$\beta$}}
\rput(3.5,5){\psframebox*[]{\textcolor{cya}{$\beta=1+\frac{1}{\alpha-2}$}}}
\rput(2.45,2.){\psframebox*[]{upper}}
\rput(3,1.65){\psframebox*[]{est.}}
\rput(3.8,1.12){\psframebox*[]{upper and lower est.}}
\def\f{1+ (1/(x-2))} 
\def\g{1+(1/(x-1))}
\psplot[plotpoints=1000,linecolor=cya]{2.25}{\xmax}{\f} 
\psplot[plotpoints=1000]{2}{\xmax}{\g}
\psline[linestyle=dashed, linecolor=red](\xmin,1)(\xmax,1) 
\psline[linestyle=dashed, linecolor=red](2,\ymin)(2,\ymax) 
\end{pspicture}
}
\end{center}

\medskip

Let us comment on the different exponents in the lower and upper estimate  of \eqref{lower}, which is valid under assumption $\beta<1+\frac{1}{\alpha-1}$. We conjecture that the correct exponent is $\frac{1}{(\alpha -1)(\beta -1)}+\frac{1}{\alpha-1}$.  To rigorously obtain the correct exponent a deeper analysis of the propagation phenomenon is needed, but this seems very involved. Indeed, the degeneracy of $f$ near zero induces a possible quenching phenomenon for the Cauchy problem. This possibility is well known for classical reaction diffusion equations \cite{Aronson1978}, \cite{Zlatos2005}, depends on $\beta$ which measures the degeneracy of $f$ at 0, and is very related to the so called {\it Fujita exponent} \cite{Fujita1966} for equation $\partial _t u=\Delta u +u^{1+p}$, $p>0$. Very recently, the Fujita exponent was  identified for the integro-differential equation $\partial _t u=J*u-u+u^{1+p}$, and the quenching phenomenon for \eqref{eq} was analyzed  \cite{Alfaro2016}.  This analysis paves the way to further studies of the acceleration in the Cauchy problem.  

\medskip

The paper is organized as follows. In Section \ref{s:tw}, we prove existence of travelling waves in the regime \eqref{alpha-beta}, that is we prove Theorem \ref{th:tw}. In Section \ref{s:notw}, we prove non existence of travelling waves in the regime \eqref{alpha-beta-notw}, that is we prove Theorem \ref{th:notw}. Last, in Section \ref{s:acceleration}, we study the acceleration phenomenon in the Cauchy problem, proving Proposition \ref{prop:level-set} $(ii)$ and Theorem \ref{th::speedacc}.

\section{Travelling waves}\label{s:tw}

In this section, we consider the regime \eqref{alpha-beta} and construct travelling waves, that is we prove  Theorem \ref{th:tw}. The main task is the contruction of a supersolution as follows.

\begin{theo}[A supersolution]\label{th:supersol} Let Assumptions \ref{ass:J} and \ref{ass:f} hold. Assume \eqref{alpha-beta}. Then we can
construct $c_0\in\R$ and a decreasing $w:\R\to (0,1)$ satisfying
$w(x)=1-e^{x}$ on  $(-\infty,-1)$, $w(x)=\frac{1}{x^{\alpha -2}}$
on  $(L,\infty)$, with $L>0$ sufficiently large, and
$$
\ep w''+J*w-w+c_0w'+f(w)\leq 0 \quad \text{ on } \R,
$$
for any $0\leq \ep \leq 1$.
\end{theo}

\begin{proof} Define a smooth decreasing function $w$ such that
$$
w(x):=\begin{cases}
1-e^x  &\text{ if } x\leq -1 \\
\displaystyle \frac{1}{ x^{p}} &\text{ if } x\geq L,\end{cases}
$$
where $L>1$ is chosen such that $1-e^{-1}>\frac{1}{L^p}$.  Since we want to show how the relation \eqref{alpha-beta} appears,
 we let $p>0$ free for the moment, and will chose
$p=\alpha -2$ only when it becomes necessary.

Notice that --- in view of \eqref{hyp-f-zero} and \eqref{hyp-f-un}--- we can find some large $r>0$ such that
$$
f(w)\leq rw^{\beta}(1-w)=:g(w), \quad \forall w\in
\left[0,1\right],
$$
so it enough to prove $\ep w''+J*w-w+c_0w'+g(w) \leq 0$ on $\R$.

\medskip

\noindent{\bf Supersolution for $x>>1$.}  Here, we work for $x\geq
2L$. Write
\begin{eqnarray*}
J*w(x)&=&\int _{-\infty}^{-1} \frac{J(y)}{(x-y)^p}dy+\int
_{-1}^{L}\frac{J(y)}{(x-y)^p}dy\\
&&+\int _{L}^{x-L} \frac{J(y)}{(x-y)^p}dy+\int _{x-L}^{\infty}
J(y)w(x-y)dy=:I_1+I_2+I_3+I_4.
\end{eqnarray*}
In the sequel, $\ep _0(z)$, $\ep (x)$ denote functions which tend
to 0 as $z\to 0$, $x\to\infty$ respectively, and which may change
from place to place. We estimate below the terms $I_k$, $1\leq
k\leq 4$, as $x\to\infty$.

$\bullet$ We use the change of variable $z=y/x$ in $I_1$ and get
\begin{eqnarray*}
I_1&=& \frac{1}{x^{p-1}}\int _{-\infty} ^{-1/x}\frac {J(xz)}{(1-z)^{p}}dz\\
&=& \frac{1}{x^{p-1}}\int _{-\infty} ^{-1/x}J(xz)(1+pz(1+\ep _0(z)))dz,
\end{eqnarray*}
where we notice that $\ep _0(z)$ remains bounded as $z\to
-\infty$, and that $\ep _0(z)\sim \frac{p+1}{2}z$ as $z\to 0$.  In view of the control \eqref{hyp-queue-gauche} of the
left tail of the kernel $J$, we can therefore cut into three
pieces, use the change of variable $y=xz$ in the first two terms, and get
\begin{eqnarray*}
I_1&=&\frac{1}{x^{p}} \int _{-\infty}^{-1} J(y)dy+\frac
p{x^{p+1}}\int_{-\infty}^{-1}yJ(y)dy+\frac p{x^{p-1}}\int
_{-\infty}^{-1/x} J(xz)z\ep_0(z)dz\\
&\leq & \frac{1}{x^{p}} \int _{-\infty}^{-1} J(y)dy+\frac
p{x^{p+1}}\int_{-\infty}^{-1}yJ(y)dy+ \frac {Cp}{x^{p+\mu -1}}\int
_{-\infty}^{-1/x}\frac 1{|z|^{\mu -1}}\vert \ep_0(z)\vert dz.
\end{eqnarray*}
Notice that $\frac{1}{\vert \cdot\vert ^{\mu -1}}\ep _0(\cdot) \in L^1(-\infty,-1)$ since $\mu>2$, and that $\frac 1{|z|^{\mu -1}}\vert \ep_0(z)\vert \sim \frac{p+1}{2}\frac{1}{\vert z\vert ^{\mu -2}}$ as $z\to 0$. On the one hand, if $2<\mu<3$ then  $\frac{1}{\vert \cdot\vert ^{\mu -1}}\ep _0(\cdot) \in L^1(-1,0)$, so that
$$
\int _{-\infty}^{-1/x}\frac 1{|z|^{\mu
-1}}\vert \ep_0(z)\vert dz=\ep(x)\frac 1{x^{2-\mu}}$$
 holds clearly. On the other hand, if $\mu\geq 3$ then  $\frac{1}{\vert \cdot\vert ^{\mu -1}}\ep _0(\cdot) \notin L^1(-1,0)$, so that
$$
\int _{-\infty}^{-1/x}\frac 1{|z|^{\mu
-1}}\vert \ep_0(z)\vert dz\sim \frac{p+1}{2}\int _{-\infty}^{-1/x}\frac 1{|z|^{\mu -2}}
dz=\ep(x)\frac 1{x^{2-\mu}},
$$
as $x\to\infty$. In any case, we conclude that
\begin{equation}\label{I1}
I_1\leq \frac 1{x^p}\int _{-\infty}^{-1} J(y)dy+\frac
p{x^{p+1}}\left[\int_{-\infty}^{-1}yJ(y)dy+\ep(x)\right].
\end{equation}

$\bullet$ For the term $I_2$, we use the same arguments to first
obtain
$$
I_2\leq  \frac{1}{x^{p}} \int _{-1}^L J(y)dy+\frac
p{x^{p+1}}\int_{-1}^{L}yJ(y)dy+ \frac {p}{x^{p -1}}\int
_{-1/x}^{L/x}J(xz)z\ep_0(z)dz.
$$
Next, since
$$
\int _{-1/x}^{L/x}J(xz)z\ep_0(z)dz=\frac 1{x^2}\int_{-1}^L J(y)\ep
_0\left(\frac y x\right)dy=\frac 1{x^2}\ep(x)\int _{-1}^ L J(y)dy,
$$
we conclude that
\begin{equation}\label{I2}
I_2\leq \frac 1{x^p}\int _{-1}^{L} J(y)dy+\frac
p{x^{p+1}}\left[\int_{-1}^{L}yJ(y)dy+\ep(x)\right].
\end{equation}

$\bullet$ We use the change of variable $z=y/x$ in $I_3$ and get
\begin{eqnarray}
I_3&=&
\frac{1}{x^{p-1}}\int_{L/x}^{1-L/x}\frac{J(xz)}{(1-z)^p}dz\nonumber \\
&=&\frac 1{x^{p-1}}\left[\int
_{L/x}^{1/2}\frac{J(xz)}{(1-z)^p}dz+\int
_{1/2}^{1-L/x}\frac{J(xz)}{(1-z)^p}dz\right]\nonumber \\
&=&\frac 1{x^{p-1}}\left[\int
_{L/x}^{1/2}\frac{J(xz)}{(1-z)^p}dz+\int
_{L/x}^{1/2}\frac{J(x(1-u))}{u^p}du\right].\label{derniereligne}
\end{eqnarray}
Using the same arguments as above, the first term in the bracket above is recast as
$$
\int _{L/x}^{1/2}\frac{J(xz)}{(1-z)^p}dz=\frac 1 x \int_L
^{x/2}J(y)dy+\frac p{x^2}\left[\int _L
^{x/2}yJ(y)dy+\int _L
^{x/2}yJ(y)\ep _0\left(\frac y x \right)dy\right],
$$
where $\ep _0(z)$ remains bounded as $z\to \infty$. Since $y\mapsto yJ(y) \in L^{1}(1,\infty)$ it follows from the dominated convergence theorem that $\int _L
^{x/2}yJ(y)\ep _0\left(\frac y x \right)dy\to 0$, as $x\to \infty$, so that
$$
\int _{L/x}^{1/2}\frac{J(xz)}{(1-z)^p}dz=\frac 1 x \int_L
^{x/2}J(y)dy+\frac p{x^2}\left[\int _L
^{x/2}yJ(y)dy+\ep(x)\right].
$$
For the second term in \eqref{derniereligne} we use the control \eqref{hyp-queue-droite} of
the right tail of the kernel  $J$ to collect
\begin{eqnarray*}
\int _{L/x}^{1/2}\frac{J(x(1-u))}{u^p}du&\leq& \frac
C{x^\alpha}\int _{L/x}^{1/2} \frac 1{u^p(1-u)^\alpha}du
\\
&\leq & \frac{C2^\alpha}{x^\alpha}\int _{L/x}^{1/2} \frac
1{u^p}du\\
&=& \frac{C 2^\alpha}{x^\alpha} \frac 1{x^{2-\alpha}}\ep(x)=\frac
1 {x^2}\ep(x),
\end{eqnarray*}
if we further assume that $0<p<\alpha -1$. As a result, we collect
\begin{equation}\label{I3}
 I_3\leq \frac 1 {x^p} \int_L
^{x/2}J(y)dy+\frac p{x^{p+1}}\left[\int _L
^{x/2}yJ(y)dy+\ep(x)\right].
\end{equation}

$\bullet$ For $I_4$ we use the crude estimate $w\leq 1$ and the
control \eqref{hyp-queue-droite} of the right tail of the kernel
$J$ to obtain
\begin{equation}\label{I4}
I_4 \leq \int _{x-L}  ^{\infty} J(y)dy\leq \int _{x-L}^\infty
\frac C{y^\alpha}dy=\frac{C}{\alpha -1}\frac{1}{(x-L)^{\alpha
-1}}.
\end{equation}

$\bullet$ Summing \eqref{I1}, \eqref{I2}, \eqref{I3} and
\eqref{I4} we arrive at
\begin{eqnarray*}
J*w(x)&\leq& \frac 1 {x^p} \int_{-\infty} ^{x/2}J(y)dy+\frac
p{x^{p+1}}\left[\int _{-\infty}
^{x/2}yJ(y)dy+\ep(x)\right]+\frac{C}{\alpha
-1}\frac{1}{(x-L)^{\alpha -1}}\\
&\leq & \frac 1 {x^p} +\frac
p{x^{p+1}}\left[J_1+\ep(x)\right]+\frac{C}{\alpha
-1}\frac{1}{(x-L)^{\alpha -1}},
\end{eqnarray*}
since $\int _\R J=1$, and where $J_1=\int_{\R} yJ(y)dy$. As a
consequence we have, for any $0\leq \ep \leq 1$, (recall that
$g(w)=rw^\beta(1-w)\leq rw^\beta$)
\begin{eqnarray*}
&&\ep w''(x)+J*w(x)-w(x)+c_0w'(x)+g(w(x))\\
&&\leq\frac{ p(p+1)}{x^{p+2}}   -\frac
p{x^{p+1}}(c_0-J_1+\ep(x))+\frac{C}{\alpha
-1}\frac 1{(x-L)^{\alpha
-1}}+\frac r{x^{p\beta}}\\
&&\leq -\frac p{x^{p+1}}(c_0-J_1+\ep(x))+\frac{C}{\alpha
-1}\frac
1{(x-L)^{\alpha -1}}+\frac r{x^{p\beta}}.
\end{eqnarray*}
For the right-hand side member to be nonpositive for large
positive $x$, one needs $p+1\leq \alpha -1$ and $p+1\leq p\beta$,
that is $\frac 1{\beta -1} \leq p \leq \alpha -2$. In view of
assumption \eqref{alpha-beta}, such a choice is possible and the
optimal one is $p=\alpha -2$, so that
\begin{eqnarray*}
&&\ep w''(x)+J*w(x)-w(x)+c_0w'(x)+w^\beta(x)
\\
&&\leq -\frac {\alpha -2} {x^{\alpha
-1}}\left(c_0-J_1-\frac{C}{(\alpha -2)(\alpha
-1)}\frac{1}{(1-\frac L x)^{\alpha -1}}-\frac r{\alpha -2}+\ep(x)\right).
\end{eqnarray*}
Choosing
$$
c_0>c_{right}:=J_1+\frac{C}{(\alpha-2)(\alpha
-1)}+\frac r{\alpha -2},
$$
we conclude that there is $M>2$ large enough so that, for any
$0\leq \ep \leq 1$,
$$
\ep w''(x)+J*w(x)-w(x)+c_0w'(x)+g(w(x))\leq 0,\quad \forall x\geq
M.
$$

\medskip

\noindent{\bf Supersolution for $x\leq -1$.} Here, we work for $x\leq -1$. The non degeneracy of 1
makes the analysis easy \lq\lq on the left''. Using the crude estimate $J*w\leq 1$, we get
\begin{eqnarray*}
&&\ep w''(x)+J*w(x)-w(x)+c_0w'(x)+g((w(x))\\
 &&\qquad \leq  -\ep e^x+1-(1-e^{x})-c_0e^{x}+r(1-e^{x})^{\beta}e^{x}\\
&&\qquad \leq  -e^{x}(c_0-1-r).
\end{eqnarray*}
Choosing $
c_0>c_{left}:=1+r$, we get, for any $0\leq\ep \leq 1$,
$$
\ep w''(x)+J*w(x)-w(x)+c_0 w'(x)+g(w(x))\leq 0,\quad \forall x\leq -1.
$$

\medskip

\noindent{\bf Supersolution everywhere.} We now finalize our choices. For $p=\alpha -2$, we define $w(x)$ as above. Then we choose a speed
$$
c_0>\max\left(c_{right},c_{left},c_{middle}:=\frac{(\max_{x\in\left[-1,M\right]} {w''(x)})^++1+r}{\min_{x\in\left[-1,M\right]} {-w'(x)}}\right).
$$
It follows from the above computations that, for any $0\leq \ep\leq 1$,
$$
\ep w''+J*w-w+c_0w'+g(w)\leq 0,
$$
holds true in $(M,\infty)$, $(-\infty,-1)$ but also in $\left[-1,M\right]$ thanks to the crude estimates $\ep\leq 1$, $J*w-w\leq 1$, $g(w)\leq r$ and $c_0>c_{middle}$. The theorem is proved.
\end{proof}

In view of \cite[Theorem 1.3]{Coville2008a}, the construction of a supersolution in Theorem \ref{th:supersol} is enough to ensure the existence of travelling waves. More precisely, there is $c^{*}\leq c_0$ such that for all $c\geq c^{*}$ equation \eqref{eq} admits travelling waves $(c,u)$, whereas, for all $c<c^{*}$ equation \eqref{eq} does not admit travelling wave. 

To complete the  proof of Theorem \ref{th:tw}, it remains to prove that $c^{*}>J_1$, which we do in the following a priori estimate on travelling wave. Notice that assumption \eqref{alpha-beta} is not required in Lemma \ref{lem:integrability}, so that its results remain valid in the regime \eqref{alpha-beta-notw} where there is no travelling wave, see Section \ref{s:notw}.

\begin{lem}[Speed from below and integrability property]\label{lem:integrability} Let Assumptions \ref{ass:J} and \ref{ass:f} hold. Let $(c,u)$ be a travelling wave.  Then $c>J_1=\int _\R yJ(y)dy$, and
\begin{equation}
\label{integrabilite}
\int _0 ^{+\infty}(u(x))^{\beta}dx<+\infty.
\end{equation}
\end{lem}

\begin{proof} We first claim (see below for a proof) that 
\begin{equation}
\label{claim}
 I_R:=\int _{-R}^{R}(J*u-u)(x)dx \to J_1, \quad \text { as } R\to \infty.
\end{equation}
Combining this with the travelling wave equation, $cu'\in L^1(\R)$ and $f(u)\geq 0$ we get that $f(u)\in L^1(\R)$, which in turn implies \eqref{integrabilite} since $f(u(x))\sim ru^\beta (x)$ as $x\to \infty$. Also integrating the travelling wave equation on $\R$, we find
$$
c-J_1=\int _{\R}f(u(x))dx>0,
$$
and it only remains to prove the claim  \eqref{claim}. 

Let us first assume that the wave $u$ is in $W^{1,\infty}(\R)$ --- which is in particular the case as soon as $c\neq 0$--- so that one can write
$$
I_R=\int _{-R}^R \int _\R J(y)(u(x-y)-u(x))dydx
=\int _{-R}^R \int _\R \int _0  ^{1}J(y)u'(x-ty)(-y)dtdydx.
$$
The absolute value of the integrand is bounded by $\Vert u'\Vert _\infty \vert y\vert J(y)$ which belongs to $L^{1}(\R)$, so that Fubini's theorem yields
$$
I_R=-\int _\R yJ(y)\int _0 ^{1}\int _{-R}^{R}u'(x-ty)dxdtdy,
$$
and thus
\begin{equation}\label{egalite}
I_R=-\int _\R yJ(y)\int _0 ^{1}(u(R-ty)-u(-R-ty))dtdy.
\end{equation}
Now the boundary conditions $u(\infty)=0$, $u(-\infty)=1$ and the dominated convergence theorem yields \eqref{claim}. 

It therefore only remains to consider the $c=0$ case, for which we only know $u\in L^\infty(\R)$. We use a mollifying argument. Let $(\rho _n)_{n\geq 0}$ be a sequence of mollifiers and define $u_n:=\rho _n *u\in C^\infty(\R)$, so that $\Vert u_n\Vert _\infty \leq \Vert u\Vert _\infty =1$. Up to an extraction, $u_n\to u$ almost everywhere on  $\R$, and, by the dominated convergence theorem, $
J*u_n\to J*u$ on  $\R$. Using again the dominated convergence theorem, we see that
$$
I_R^{n}:=\int _{-R}^{R}(J*u_n-u_n)(x)dx\to I_R, \; \text{ as } n\to \infty.
$$
On the other hand, for a given $n\geq 0$, $u_n\in W^{1,\infty}(\R)$ so that equality \eqref{egalite} applies  to $u_n$ and
\begin{eqnarray*}
I_R^{n}&=&-\int _\R yJ(y)\int _0 ^{1}(u_n(R-ty)-u_n(-R-ty))dtdy\\
 &&\to -\int _\R yJ(y)\int _0 ^{1}(u(R-ty)-u(-R-ty))dtdy,\; \text{ as } n\to \infty.
\end{eqnarray*}
{}From the above $n\to \infty$ limits of $I_R^n$, we get that equality \eqref{egalite} is still true in this $c=0$ case, and we conclude as above.
\end{proof}

\section{Non existence of travelling wave}\label{s:notw}

In this section, we consider the regime \eqref{alpha-beta-notw} and prove non existence of travelling wave, that is we prove  Theorem  \ref{th:notw}. 

We begin with an estimate of the nonlocal diffusion term for an algebraic tail which will then serve, twice, as a subsolution. This estimate is in the spirit of \cite{Mellet2014}, where the nonlocal diffusion operator is the fractional Laplacian. Nevertheless, since our nonlocal diffusion operator does not share the homogeneity property (allowed by some singularity in zero) of the fractional Laplacian, we need to deal with an additional bad negative term in \eqref{tail-nonlocal}.

\begin{lem}[Estimate for an algebraic tail]\label{lem:mellet} Let Assumption \ref{ass:J-noTW} hold. For $p>0$, let us define
 $$
w(x):=\begin{cases}
1  &\text{ if } x<1 \\
\displaystyle \frac{1}{ x^{p}} &\text{ if } x\geq 1.\end{cases}
$$
Then there are constants $C_1>0$ and $C_2>0$ such that
\begin{equation}\label{tail-nonlocal}
J*w(x)-w(x)\geq \frac{C_1}{x^{\alpha-1}}-\frac{C_2}{x^{p+1}}+O\left(\frac{1}{x^{p+\alpha -1}}\right), \quad \text{ as } x\to +\infty.
\end{equation}
\end{lem}

\begin{proof} For $x\geq 2$, we write
$$
J*w(x)-w(x)=\int _{-\infty}^{x-1}J(y)\left(\frac{1}{(x-y)^{p}}-\frac{1}{x^p}\right)dy+\int _{x-1}^{+\infty}J(y)dy
+\int _{x-1}^{+\infty}J(y)\frac{-1}{x^p}dy.
$$
In view of \eqref{hyp-queue-droite-noTW}, the second integral term above is larger than $\frac{C_1}{x^{\alpha-1}}$ for some $C_1>0$, whereas the third integral term is
$O\left(\frac{1}{x^{p+\alpha -1}}\right)$ as $x\to +\infty$. In the first integral term, we perform the change of variable $y=xz$ and cut into three pieces to obtain $$
J*w(x)-w(x)=I_1+I_2+I_3+\frac{C_1}{x^{\alpha-1}}+O\left(\frac{1}{x^{p+\alpha -1}}\right),
$$
where

$\bullet$ 
$
I_1:=\frac{1}{x^{p-1}}\int_{-\infty}^{-1/x}J(xz)\left(\frac{1}{(1-z)^{p}}-1\right)dz$. In view of \eqref{hyp-queue-gauche}, we obtain
$$
\vert I_1\vert \leq \frac{C}{x^{p+\mu-1}}\int _{-\infty}^{-1/x}\frac{1}{\vert z\vert ^{\mu}}\left(1-\frac{1}{(1-z)^{p}}\right)dz.
$$
Since the above integrand is equivalent to $\frac{p}{\vert z\vert ^{\mu-1}}$ as $z\to 0$ (with $\mu -1>1$), we end up  with $\vert I_1\vert =O\left(\frac{1}{x^{p+1}}\right)$ as $x\to +\infty$.

$\bullet$
$
I_2:=\frac{1}{x^{p-1}}\int_{-1/x}^{1/x}J(xz)\left(\frac{1}{(1-z)^{p}}-1\right)dz$. Hence
$$
\vert I_2\vert \leq \frac{\Vert J\Vert _\infty}{x^{p-1}}\int _{-1/x}^{1/x}\left\vert \frac{1}{(1-z)^{p}}-1 \right\vert dz.
$$
Since the above integrand is equivalent to $p\vert z\vert$ as $z\to 0$, we end up with $\vert I_2\vert =O\left(\frac{1}{x^{p+1}}\right)$ as $x\to +\infty$.

$\bullet$
$I_3:=\frac{1}{x^{p-1}}\int_{1/x}^{1-1/x}J(xz)\left(\frac{1}{(1-z)^{p}}-1\right)dz
\geq 0$. 
 
Putting all together concludes the proof of the lemma.
\end{proof}

We can now prove below some a priori algebraic estimates of the tails of possible travelling waves. 

\begin{lem}[A priori estimates of tails from below]\label{lem:below}  Let Assumptions \ref{ass:J-noTW} and \ref{ass:f} hold. Let $\ep >0$ be given. Then for any travelling wave $(c,u)$, there is a constant $K>0$ such that
$$
 u(x)\geq \frac{K}{x^{\alpha-2+\ep}}, \text{ for all } x\geq 1.
$$
\end{lem}

\begin{proof} For a travelling wave $(c,u)$, since $f\geq 0$ on $[0,1]$, we have
\begin{equation}\label{v-super}
J*u-u+cu'\leq 0 \; \text{ on } \R.
\end{equation}

On the other hand, Lemma \ref{lem:mellet} implies that, for $A>1$ large enough, the function $
w(x):=\mathbf 1 _{(-\infty,1]}(x)+\mathbf 1 _{(1,\infty)}(x) \frac{1}{ x^{\alpha-2+\ep}}
$
satisfies
\begin{equation}\label{w-sub}
J*w-w+cw'>0 \; \text{ on } (A,\infty).
\end{equation}

Since  $w\leq 1$ and $\inf _{(-\infty,A]}u>0$, we
can select $K >0$ small enough so that  $K w(x)-u(x)<0$ for all $x\in(-\infty,A]$. Assume by contradiction that there is  $x_0>A$ such that $K w(x_0)-u(x_0)>0$. Since $Kw(x)-u(x)\to 0$ as $x\to +\infty$, the function $Kw-u$ reaches some global maximum at some point $x_1\in (A,\infty)$, which is contradicted by \eqref{v-super} and \eqref{w-sub}. This proves the lemma. \end{proof}

The next lemma is of crucial importance for the proof of non existence of waves under  assumption \eqref{alpha-beta-notw}. Roughly speaking, if the tail of a travelling wave is rather heavy then it is actually very heavy. Notice that such a trick was also used in \cite{Gui2015}. In contrast with the previous lemma, we shall need to keep a trace of the nonlinear term to improve the tail estimate. 

\begin{lem}[Making the tail heavier]\label{lem:heavier} Let Assumptions \ref{ass:J-noTW} and \ref{ass:f} hold. Let $(c,u)$ be a travelling wave. Assume that there are 
\begin{equation}\label{tantque}
\frac 1\beta <\gamma<\frac{1}{\beta-1},
\end{equation}
 and $K>0$ such that
\begin{equation}\label{rather-heavy}
u(x)\geq \frac{K}{x^{\gamma }}, \text{ for all } x\geq 1.
\end{equation}
Then, there is $M>0$ such that
$$
 u(x)\geq \frac{M}{x^{\beta\gamma -1}},  \text{ for all } x\geq 1.
$$
\end{lem}

\begin{proof} 
Using $f(u)\sim r u ^{\beta }$ as $u\to 0$ and estimate \eqref{rather-heavy}, we deduce that there is $\delta >0$ such that $f(u(x))\geq \frac{\delta}{x^{\beta\gamma }}$ if $x$ is sufficiently large. Hence, we get the existence of $A>1$ such that
\begin{equation}\label{v-super-bis}
J*u-u+cu'+\frac{\delta}{x^{\beta\gamma}}\leq 0 \; \text{ on } (A,\infty).
\end{equation}

On the other hand, since $p:=\beta\gamma -1>0$, it follows from Lemma \ref{lem:mellet} that the function $
w(x):=\mathbf 1 _{(-\infty,1]}(x) +\mathbf 1 _{[1,\infty)}(x) \frac{1}{ x^{\beta\gamma -1}}
$
satisfies
$$
J*w(x)-w(x)+cw'(x)\geq \frac{C_1}{x^{\alpha-1}}-\frac{C_2'}{x^{\beta\gamma}}+O\left(\frac{1}{x^{\beta\gamma +\alpha-2}}\right),
$$
as $x\to +\infty$, where $C_2'=C_2+\vert c\vert (\beta\gamma -1)>0$.

Therefore, for any $0<M<\frac{\delta/2}{C_2'}$, the function $z:=Mw-u$ satisfies 
\begin{equation}
\label{z-sub}
J*z(x)-z(x)+cz'(x)\geq \frac{\delta /2}{x^{\beta\gamma}}+O\left(\frac{1}{x^{\beta\gamma+\alpha-2}}\right)>0 \; \text{ on } (A,\infty),
\end{equation}
up to enlarging $A>1$ if necessary, 

We conclude as in Lemma \ref{lem:below}:  we
can select  $0<M<\frac{\delta/2}{C_2'}$  so that $z(x)=M w(x)-u(x)<0$ for all $x\in(-\infty,A]$. Assume by contradiction that there is  $x_0>A$ such that $z(x_0)>0$. Since $z(x)=Mw(x)-u(x)\to 0$ as $x\to +\infty$, the function $z$ reaches some global maximum at some point $x_1\in (A,\infty)$, which is contradicted by \eqref{z-sub}. This proves the lemma. 
\end{proof}

Equipped with the above a priori estimates and the integrability property \eqref{integrabilite} of Lemma \ref{lem:integrability}, we can now prove Theorem \ref{th:notw}.

\begin{proof}[Proof of Theorem \ref{th:notw}] Let us assume by contradiction inequality \eqref{alpha-beta-notw} together with the existence of a travelling wave $(c,u)$.

\noindent{\bf First case:} $0<\alpha -2<\frac 1 \beta$. In this regime, Lemma \ref{lem:below} and Lemma \ref{lem:integrability} are enough to derive a contradiction. Indeed, we select $\ep>0$ small enough so that $0<\alpha-2+\ep<\frac{1}{\beta}$. It follows from Lemma \ref{lem:below} that 
$u(x)^{\beta}\geq \frac{K^{\beta}}{x^{\beta(\alpha-2+\ep)}}$, for all $x\geq 1$. Since $\beta(\alpha-2+\ep)<1$, this contradicts the integrability property \eqref{integrabilite}.

\noindent{\bf Second case:} $\frac 1 \beta \leq \alpha -2<\frac 1 {\beta-1}$. In this regime, we further need to iterate Lemma \ref{lem:heavier} to derive a contradiction. We first select $\ep >0$ small enough so that $\frac 1 \beta<\gamma:=\alpha -2+\ep<\frac{1}{\beta-1}$. It follows from Lemma \ref{lem:below} that the assumptions of Lemma \ref{lem:heavier}
 hold true, so that we can apply it once (at least).  Notice that the recursive sequence
 $$
 \gamma _0=\gamma, \, \gamma _{n+1}=\beta\gamma _n -1
 $$
 tends to $-\infty$ as $n\to+\infty$, so that there is $N\geq 1$ such that
 $$
 \gamma _{N}\leq \frac 1\beta<\gamma _{N-1}<\cdots<\gamma _0<\frac{1}{\beta-1}.
 $$
 This allows us to apply recursively Lemma \ref{lem:heavier} $N$ times and to end up with a constant $D>0$ such that $u(x)\geq \frac{D}{x^{\gamma _{N}}}$, for all $x\geq 1$. Since $\beta\gamma _{N}\leq 1$, this again contradicts the integrability property \eqref{integrabilite}.
 
 Theorem \ref{th:notw} is proved. \end{proof}

\section{Acceleration in the Cauchy problem}\label{s:acceleration}

Through this section we assume \eqref{alpha-beta-notw} and study the acceleration phenomenon arising in the Cauchy problem
\begin{align}
&\partial_t u(t,x)=J* u(t,x) -u(t,x) +f(u(t,x)) \quad  t>0,\, x\in \R, \label{ac-eq-cauchy1}\\
&u(0,x)=u_0(x) \quad x\in \R,\label{ac-eq-cauchy2}
\end{align}
when $u_0$ is a front like initial data, in the sense of Assumption \ref{ass:u0}. 

\subsection{Proof of acceleration}

Here we prove Proposition \ref{prop:level-set} $(ii)$.

To do so, we need a preliminary result on ignition problems that serve as an approximation of our degenerate monostable problem. For $0<\theta <1$ we consider a smooth {ignition nonlinearity} $g_\theta:[0,1]\to \R$, meaning $g_\theta=0$ on $[0,\theta]\cup \{1\}$, $g_\theta >0$ on $(\theta,1)$. As proved in \cite{Coville2007b}, there is a unique speed $c_\theta \in \R$ and a unique decreasing profile $U_\theta$ solving the travelling wave problem
\begin{align}
J* U_\theta-U_\theta+c_\theta U_\theta'+g_\theta(U_\theta)=0 \quad \text{ on } \R,\label{eq:igni}\\
U_\theta(-\infty)=1, \quad U_\theta(0)=\theta, \quad U_\theta(\infty)=0.\label{eq:igni-lim}
\end{align}
Notice that when $c_\theta\neq 0$, $U_\theta\in C^1$ and  satisfies the equation in the classical sense. On the other hand, when $c_\theta=0$, depending on $g_\theta$ the function $U_\theta$ may have some discontinuities. However, in such a situation \eqref{eq:igni} is satisfied almost everywhere and the limits and the normalisation \eqref{eq:igni-lim} are still valid. As a consequence of the non existence of monostable waves 
Theorem \ref{th:notw}, we can prove the following.

\begin{prop}[Speeds of a sequence of  ignition waves]\label{prop:cinfinie}
Let Assumptions \ref{ass:J-noTW} and \ref{ass:f} hold.   Assume  \eqref{alpha-beta-notw}. Let $(g_{n})=(g_{\theta _n })$ be a sequence of ignition nonlinearities such that  $g_n\le  g_{n+1}\le f$ and $g_n\to f$. Let $(c_n,U_n)$ be the associated sequence of travelling waves. Then
$$
\lim_{n\to \infty }c_n=+\infty.
$$
\end{prop}

\begin{proof} Since $g_{n+1}\geq g_n$ it follows from \cite[Corollary 2.3]{Coville2008a} that $c_{n+1}\geq c_n$. Assume by contradiction that $c_n \nearrow \bar c$ for some $\bar c\in \R$.  We distinguish two cases.

Assume here $\bar c\neq 0$. There is thus an integer $n_0$  such that, for all $n\ge n_0$, $c_n\neq 0$. As a consequence, for all $n\ge n_0$, $U_n$ is smooth and since any translation of $U_n$ is a still a solution, without any loss of generality we can assume  the normalisation $U_n(0)=1/2$. Now,  thanks to Helly's Theorem \cite{Brunk1956} and up to extraction, $U_n$ converges to a monotone function $\bar U$ such that $\bar U(0)=\frac 12$. Also, from the equation and up to extraction, $U_n$ also  converges in  $C^1_{loc}(\R)$, and the limit has to be $\bar U$. As a result, $\bar U$ is monotone and solves 
$$\begin{cases}
J*\bar U-\bar U+\bar c\bar U'+f(\bar U)=0 \quad \text{ on } \R,\\
\bar U(-\infty)=1,\quad \bar U(0)=\frac{1}{2}, \quad \bar U(\infty)=0.
\end{cases}
$$ 
In other words, we have contructed a monostable travelling wave under assumption \eqref{alpha-beta-notw}, which is  a contradiction with Theorem \ref{th:notw}. 

Assume here $\bar c=0$. Since $(c_n)$ is nondecreasing, either $c_n<0$ for all $n$, either there is an integer $n_0$ such that $c_n=0$ for all  $n\geq n_0$. In the former case,  since for all $n$ $U_n$ is smooth, without loss of generality $U_n$ can be normalized by $U_n(0)=\frac{1}{2}$.   We can then use the Helly's Theorem \cite{Brunk1956} and the normalisation to pass to the limit in the equation in a weak sense to obtain a monotone function $\bar U$ such that 
$$
\begin{cases}
J*\bar U-\bar U+f(\bar U)=0 \quad \text{almost everywhere in }  \R,\\
\bar U(-\infty)=1,\quad \bar U(0)=\frac{1}{2}, \quad \bar U(\infty)=0,
\end{cases}
$$ 
which again  contradicts  Theorem \ref{th:notw}. Let us now consider the remaining case, $c_n=0$ for $n\ge n_0$. Observe that from Assumption \ref{ass:f} we can find $0<s_0 <1$ such that  $s-f(s)$ is a one-to-one function in $[0,s_0]$  and, since $g_n\to f$ is of ignition type,  $s-g_n(s)$ is also a one-to-one function in $[0,s_0]$ for all $n$. Now since for $n\ge n_0$,   $U_n$ satisfies $U_n-g_n(U_n)=J*U_n$,  $U_n$ has to be continuous in $[U_n^{-1}(s_0),\infty)$. Now, thanks to invariance by translation,  we can assume that, for all $n\ge n_0$,   $U_n(0)=s_0$. The sequence of monotone functions $(U_n)_{n\ge n_0}$ being  bounded, thanks to Helly's Theorem \cite{Brunk1956} and the normalisation, as $n\to \infty$,  $U_n$ converges pointwise  to a monotone function $\bar U$ solution  of 
$$\begin{cases}
J*\bar U-\bar U+f(\bar U)=0 \quad \text{ on } \R,\\
\bar U(-\infty)=1,\quad \bar U(0)=s_0, \quad \bar U(\infty)=0,
\end{cases}
$$ 
which again contradicts Theorem \ref{th:notw}.
\end{proof}

\begin{rem}\label{rem:bistable}
Clearly, the results of Proposition \ref{prop:cinfinie} stand as well if we replace the ignition type nonlinearity $g_n$ by a bistable type nonlinearity.
\end{rem}
\newpage

We are now in the position to prove the first part \eqref{invasion} of Proposition \ref{prop:level-set} $(ii)$.

\begin{proof}[Proof of \eqref{invasion}] First, we prove \eqref{invasion} for the particular case where the initial datum $u_0$ is a smooth nonincreasing function such that
\begin{equation}\label{d0}
u_0(x)=\begin{cases}
d_0 &\text{ for }  x \le -1\\
0 &\text{ for }  x \ge 0,
\end{cases}
\end{equation}
for an arbitrary $0<d_0<1$. Since $u_0$ is  nonincreasing, we deduce from the comparison principle that, for all $t>0$, the function $u(t,x)$ is still decreasing in $x$.

Let us now extend $f$  by $0$ outside  the interval $[0,1]$. From \cite{Coville2007b}, Proposition \ref{prop:cinfinie} and Remark \ref{rem:bistable}, there exists $0<\theta<d_0$ and  a Lipschitz bistable function $g\le f$ --- i.e. $g(0)=g(\theta)=g(1)=0$,  $g(s)<0$ in $(0,\theta)$, $g(s)>0$ in $(\theta,1)$, and $g'(0)<0$, $g'(1)<0$, $g'(\theta)>0$---  such that there exists a smooth decreasing  function $U_\theta$ and  $c_\theta>0$ verifying
\begin{align*}
&J* U_\theta-U_\theta +c_{\theta} U_\theta'+g(U_\theta)=0 \quad \text{ on } \R,    \\
&U_\theta(-\infty)=1, \qquad U_\theta(\infty)=0.
\end{align*}

Let us now consider $v(t,x)$ the solution of the Cauchy problem
\begin{align*}
&\partial_t v(t,x) =J* v(t,x)-v(t,x) + g(v(t,x)) \quad \text{ for  } t>0, x\in \R,    \\
&v(0,x)=u_0(x).
\end{align*}
Since $g\le f$, $v$ is a subsolution of the Cauchy problem \eqref{ac-eq-cauchy1}-\eqref{ac-eq-cauchy2} and by the comparison principle, $v(t,x)\le u(t,x)$ for all $t>0$ and $x \in \R$.

Now, thanks to the global asymptotic stability result \cite[Theorem 3.1]{Chen1997}, since $d_0>\theta$  we know that there exists $\xi\in\R, C_0>0$  and $\kappa>0$ such that for all $t\ge 0$
$$
\|v(t,\cdot)-U_{\theta}(\cdot-c_\theta t+\xi)\|_{L^\infty} \le C_0 e^{-\kappa t}.
$$
Therefore, we have for all $t>0$ and $x \in \R$,
$$
u(t,x)\ge v(t,x)\ge U_{\theta}(x-c_\theta t+\xi)-C_0e^{-\kappa t}.
$$
Since $c_\theta>0$,  by sending $t \to \infty$, we get $1\ge \liminf_{t\to \infty}u(t,x)\ge \lim_{t\to \infty}U_{\theta}(x-c_\theta t+\xi)-C_0e^{-\kappa t}= 1$. As a result, for all $x$, $1\ge \limsup_{t\to \infty} u(t,x)\ge \liminf_{t\to \infty} u(t,x)= 1$, and therefore $u(t,x)\to 1$  as $t \to \infty$.  Since $u(t,x)$ is nonincreasing in $x$, the convergence is uniform on any set $(-\infty,A]$. This concludes the proof of \eqref{invasion} for our particular initial datum.

For a generic initial data satisfying Assumption \ref{ass:u0}, we can always, up to a shift in space, construct a smooth nonincreasing $\tilde{u_0}$ satisfying \eqref{d0} and $\tilde{u_0}\leq u_0$. Since the solution $\tilde u(t,x)$ of the Cauchy problem starting from $\tilde{u_0}$ satisfies \eqref{invasion}, so does $u(t,x)$ thanks to the comparison principle. 
\end{proof}

\begin{rem} Notice that the above proof only uses elementary arguments and holds as well for other types of reaction diffusion equations, as soon as a travelling front solution with a positive speed exists when the nonlinearity considered is replaced by any  nonlinearity  of ignition or bistable type. In particular,  it applies to  solutions of  Cauchy problems where the operator  $J* u -u$ is replaced by a fractional Laplacian $-(-\Delta)^{s}u$, $0<s<2$, or the standard diffusive operator $\Delta u$.  
\end{rem}

The property \eqref{invasion} now guarantees that for any $\lambda \in (0,1)$, the super level set $\Gamma_\lambda(t)$ is never empty for large time.
Let us now prove the second part \eqref{acc-prop} of Proposition \ref{prop:level-set} $(ii)$.

\begin{proof}[Proof of \eqref{acc-prop}]  Let $\lambda\in(0,1)$ be given. As above, there is no loss of generality to assume that the initial datum is as in the beginning of the proof of \eqref{invasion}, so that $u(t,x)$ is nonincreasing in $x$. From this and \eqref{invasion}, either for each $t>0$ large enough $\Gamma_\lambda(t)$ is bounded from above and $\Gamma _\lambda (t)=(-\infty,x_\lambda(t)]$,  or   $ \Gamma_\lambda(t_0)=(-\infty,\infty)$ for some $t_0>0$. In the latter situation,  using the constant $\lambda$ as a subsolution, we see that,  for all $t\ge t_0$, $ \Gamma_\lambda(t)=(-\infty,\infty)$ so that $x_\lambda(t)=+\infty$ and we are done. In the sequel, we assume that  for large $t$, says $t\ge t_1$, $x_\lambda(t)\in \R$. 

Let $g\le f$ be a smooth function such that $g(0)=g\left(\frac{1+\lambda}{2}\right)=0$, $g(s)=0$ for $s\le 0$, $g(s) \sim f(s) $ as $s\to 0$, and $g(s)>0$ for $s\in (0,\frac{1+\lambda}{2})$. For $\theta <\frac{1-\lambda}{2}$, let us consider the ignition type nonlinearity $g_{\theta}(s):=g(s-\theta)$ and let $\theta$ small, say $\theta\le \theta_0$, such that for all $\theta\le \theta_0$  there exists a smooth decreasing  function $U_\theta$ and  $c_\theta>0$ such that
\begin{align*}
&J* U_\theta-U_\theta +c_{\theta} U_\theta'+g_{\theta}(U_\theta)=0 \quad \text{ on } \R,    \\
&U_\theta(-\infty)= \frac{1+\lambda}{2}+\theta, \qquad U_\theta(0)=\theta, \qquad U_\theta(\infty)=0.
\end{align*}

Then by a straightforward computation, we see that  $\underline{U}(t,x):=U_\theta(x-c_\theta t)-\theta$ is a subsolution to equation \eqref{eq}. Notice that $\underline {U}(t,x)<\frac{1+\lambda}{2}<1$ and $\underline U (0,x)\leq 0$. Since $u(t,x)$ converges uniformly to $1$ in the set $(-\infty,0]$,  there thus exists $t_2>t_1$ such that $u(t_2,x)\ge \underline {U}(0,x)$. Hence, from the comparison principle, $u(t+t_2,x)\ge \underline {U}(t,x)=U_\theta(x-c_\theta t)-\theta$ for all $t>0$ and $x\in \R$. Denoting by $y_\theta$ the point where $U_\theta (y_\theta)=\lambda+\theta$, this in turn implies that $x_\lambda(t)\ge y_\theta +c_{\theta}(t-t_2)$ for all $t>0$. As a result, $ \liminf _{t\to \infty}\frac{x_\lambda(t)}{t}\ge c_\theta$.
The above argument being independent of $\theta\le \theta_0$  we get, thanks to Proposition \ref{prop:cinfinie},
$$ \liminf _{t\to \infty}\frac{x_\lambda(t)}{t}\ge \lim_{\theta \to 0}c_\theta=+\infty,
$$  
which concludes the proof.
\end{proof}

\subsection{Upper bound on the speed of the super level sets}

Here we prove the upper bound \eqref{upper} of Theorem \ref{th::speedacc}. To do so, we construct an adequate supersolution.

\begin{proof}[Construction of a supersolution] For $p>0$ to be specified later, let us define
\begin{equation}\label{v0}
v_0(x):=\begin{cases}
1 &\text{ if } x\leq 1 \\
\frac{1}{ x^{p}} &\text{ if } x> 1.
\end{cases}
\end{equation}
For  $\gamma>0$ to be specified later,  let  $w(\cdot,x)$ denote the solution of the Cauchy problem
$$
\frac{dw}{dt}(t,x)=\gamma w^{\beta}(t,x), \quad w(0,x)=v_0(x),
$$
that is
 $$
 w(t,x)=\frac{1}{\left(v_0^{1-\beta}(x) -\gamma(\beta -1) t\right)^{\frac{1}{\beta -1}}}.
 $$
 Notice that $w(t,x)$ is not defined for all times. When $x\le 1$, $w(t,x)$ is defined for $t\in[0,\frac{1}{\gamma(\beta-1)})$ whereas for $x>1$,   $w(t,x)$ is defined for $t\in\left[0,T(x):=\frac{x^{p(\beta -1)}}{\gamma(\beta-1)}\right)$.
 Let $ x_0(t)$ be such that 
 \begin{equation}\label{ac-def-x0}
  x_0(t):=\left[1+\gamma (\beta -1)t \right]^{\frac{1}{p(\beta-1)}}\geq 1,
 \end{equation}
 so that $w(t,x_0(t))=1$ and $w(t,x)<1$ whenever $x>x_0(t)$. Last, we define 
 \begin{equation}\label{def-m}
 m(t,x):=\begin{cases}
1 &\text{ if } x\leq x_0(t) \\
w(t,x) &\text{ if } x> x_0(t),
\end{cases}
\end{equation}
and show below that $m$ is a  supersolution of \eqref{ac-eq-cauchy1}--\eqref{ac-eq-cauchy2}, provided $p>0$ and $\gamma>0$ are appropriately chosen.

If $(t,x)$ is such that $x\leq x_0(t)$, we see that $\partial _t m(t,x)=f(m(t,x))=0$, and 
$$
\partial_t m(t,x) -\left(J*  m(t,x)-m(t,x)+f( m(t,x))\right)\ge 0,
$$
since $m\leq 1$ by construction.  Hence,  it  remains to consider the $(t,x)$ such that $t>0$ and $x>  x_0(t)$, which we consider below.

In view of Assumption \ref{ass:f}, there is $r_0>0$ such that $f(u)\leq r_0 u^\beta$ for all $0\leq u\leq 1$. By definition of $ m(t,x)$, we  have  
$\partial_t  m(t,x)=\gamma w^{\beta}(t,x)$ and $f( m(t,x))\le r_0w^\beta(t,x)$. Next, for $\gamma>\gamma_0:=r_0+1$,  let us define 
 \begin{equation}\label{ac-def-xgamma}
  x_\gamma(t):=\left[\left(\gamma-r_0\right)^{\frac{\beta-1}{\beta}}+\gamma (\beta -1)t \right]^{\frac{1}{p(\beta-1)}}>1,
 \end{equation}
so that $ x_0(t)< x_\gamma(t)$ and $w(t, x)\geq \left(\frac{1}{\gamma-r_0}\right)^{\frac{1}{\beta}}$ for $x_0(t)<x\leq x_\gamma(t)$. Thus, for  $t>0$ and  $ x_0(t)< x\le  x_\gamma(t)$,
\begin{align*}
\partial_t m(t,x) -(J*  m(t,x)-m(t,x))-f( m(t,x))&\ge (\gamma-r_0)w^{\beta}(t,x)-J*  m(t,x),\\
&\ge  (\gamma-r_0)w^{\beta}(t,x)-1,\\
&\ge 0.
\end{align*}
 Hence,  it  remains to consider the $(t,x)$ such that $t>0$ and $x>x_\gamma(t)$, which we consider below.

Let us  estimate more precisely $J*  m(t,x)$ in the region $x>  x_\gamma(t)$. To simplify the presentation,  let us introduce the notations $q:=p(\beta-1)$ and $\sigma:=\gamma(\beta-1)t$. Let $K>1$ to be specified later. We write
\begin{equation}
\begin{array}{lcc} 
J*  m(t,x)&=\underbrace{\int_{-\infty}^{\frac{ x_0(t)-x}{K}}J(-z) m(t,x+z)\,dz}&+\underbrace{\int_{\frac{ x_0(t)-x}{K}}^{\infty}J(-z)\frac{1}{\left[(x+z)^{q}-\sigma\right]^{\frac{p}{q}}}\,dz}.\label{bidule}\\
&I_1&I_2
\end{array}
\end{equation}
 In view of \eqref{ac-def-xgamma}, we can select $\gamma_1=\gamma _1(K)>\gamma_0$ large enough so that, for all $\gamma\geq \gamma _1$, all $x>x_\gamma (t)$, we have $  x_0(t)-x<-K$. Therefore, from $m\leq 1$ and \eqref{hyp-queue-droite},  we get (in the sequel $C$ denotes a generic positive constant that may change from place to place)
$$
I_1\le \int_{-\infty}^{\frac{ x_0(t)-x}{K}}J(-z)\,dz\le\frac{CK^{\alpha-1}}{(x- x_0(t))^{\alpha-1}}. $$
 By choosing $q<1$ and using the definition of $ x_0(t)$ we see that 
 $$ 
 \frac{1}{(x- x_0(t))^{\alpha-1}}\le \frac{1}{(x^q- x^q_0(t))^{\frac{\alpha-1}{q}}}\le w^{\frac{\alpha-1}{p}}\left(t,\left(x^q-1\right)^{1/q}\right).
 $$
Using that $q<1$ and $w(t,\cdot)$ is a decreasing function in $(x_0(t),\infty)$ (this can be seen by computing $\partial _xw= v_0' v_0^{-\beta}w^{\beta}\leq 0$), 
we have for $x>>1$, say $x>A_0>\frac{2}{q}+1$,
$$
w^{\frac{\alpha-1}{p}}\left(t,\left(x^q-1\right)^{1/q}\right)\le w^{\frac{\alpha-1}{p}}\left(t,x-\frac{2}{q}\right).
$$
Up to enlarging $A_0$, for $x\ge A_0$ we have $4x^{q-1}<\frac{1}{2}$ and $x^q-4x^{q-1}\le \left(x-\frac{2}{q}\right)^{q}\le x^q-x^{q-1} $.  Then for such $A_0$, we see that, for $x\ge x_0(t)+A_0$,  
\begin{align*}
  \frac{w(t,x-\frac{2}{q})}{w(t,x)}= \left(\frac{1}{1-w^{\beta -1}(t,x)\left[x^q-\left(x-\frac{2}{q}\right)^q\right]}\right)^{\frac{p}{q}} &\le\left(\frac{1}{1-4x^{q-1}w^{\beta -1}(t,x)}\right)^{\frac{p}{q}},\\
  &\le\left(1+\frac{4x^{q-1}w^{\beta-1}(t,x)}{1-4x^{q-1}w^{\beta -1}(t,x)}\right)^{\frac{p}{q}},\\
  &\le 2^{\frac{1}{\beta-1}}.
  \end{align*} 
 Therefore, for $\gamma$ large enough, say $\gamma\ge \gamma_2(A_0)$,  we have for all $t> 0$,  $x\ge  x_\gamma(t)\ge \  x_0(t)+A_0$,  
\begin{equation}\label{ac-eq-I1}
I_1\le C_1K^{\alpha-1} w^{\frac{\alpha-1}{p}}(t,x).
\end{equation}

We now turn to $I_2$. Using the change of variable $u=\frac{z}{x}$  and rearranging the terms, we get    
\begin{equation}\label{termeI2}
I_2=xw(t,x)\int_{\frac{ x_0(t)}{Kx}-\frac{1}{K}}^{\infty}J(-xu)\frac{1}{\left(\frac{[1+u]^q-1}{1-\frac{\sigma}{x^q}}+1\right)^{p/q}}\,du=xw(t,x)(I_3+I_4),
\end{equation}
where
$$
I_3:= \int_{-\frac{1}{x}}^{\infty}J(-xu)\frac{1}{\left(\frac{[1+u]^q-1}{1-\frac{\sigma}{x^q}}+1\right)^{p/q}}\,du,\quad  I_4:=\int_{\frac{ x_0(t)}{Kx}-\frac{1}{K}}^{-\frac{1}{x}}J(-xu)\frac{1}{\left(\frac{[1+u]^q-1}{1-\frac{\sigma}{x^q}}+1\right)^{p/q}}\,du,
$$
which we estimate below.

 For $I_3$, since $u\in \left[-\frac{1}{x},\infty\right)$,  $q<1$ and $(1+u)^q$ is a monotone increasing function,  by using the definition of $w(t,x)$ we have 
$$
\frac{1}{\left(\frac{[1+u]^q-1}{1-\frac{\sigma}{x^q}}+1\right)^{p/q}}\le \frac{1}{\left(1-w^{\beta-1}(t,x)\right)^{p/q}}.
$$
Now, we know that for $x>x_\gamma(t)$, $w(t,x)\leq\left( \frac{1}{\gamma-r_0}\right)^{1/\beta}<1$ so that a Taylor expansion yields a constant $\bar C(q)>0$ such that
$$
 \frac{1}{\left(1-w^{\beta-1}(t,x)\right)^{p/q}}\leq 1+\frac p q(1+\bar C(q))w^{\beta -1}(t,x),
$$
so that
\begin{equation}\label{ac-eq-I3}
I_3\le \frac{1}{x}\int_{-1}^{+\infty}J(-z)\,dz + \frac{p}{q}(1+\bar C(q))\frac{w^{\beta-1}(t,x)}{x}\int_{-1}^{+\infty}J(-z)\,dz.
\end{equation}

 For $I_4$, use the following claim, whose proof is postponed. 
\begin{cla}\label{ac-cla-I4} For  $q<1$,   
there exists $K(q)>0$ such that  for all $t>0$, all $K\geq K(q)$, all $x>x_0(t)$, 
all $u\in \left[-\frac{1}{K}+\frac{ x_0(t)}{Kx},-\frac{1}{x}\right]$, we have 
$$\frac{[1+u]^q-1}{1-\frac{\sigma}{x^q}}\ge -\frac{1}{2}.$$
\end{cla}
\noindent For $q<1$, we select $K\ge K(q)$ and  $\gamma\ge \max\{\gamma_0,\gamma_1(K),\gamma_2(A_0)\}$. From the above claim, we deduce 
from  Taylor expansion of the fraction $\frac{1}{(1-z)^{\frac{p}{q}}}$, there exists a constant $\tilde C(q)$ such that 
 $$\
\frac{1}{\left(\frac{[1+u]^q-1}{1-\frac{\sigma}{x^q}}+1\right)^{p/q}}\le 1+\frac{p}{q}(1+\tilde C(q))x^qw^{\beta-1}(t,x)\left(1-[1+u]^q\right).
$$
Since $q<1$, and $J(z)|z|\in L^{1}(\R)$ it follows that 
\begin{equation}\label{ac-eq-I4}
I_4\le \frac{1}{x}\int_{\frac{ x_0(t)-x}{K}}^{-1}J(-z)\,dz + \frac{p}{q}(1+\tilde C(q))\frac{w^{\beta-1}(t,x)}{x}\int_{\frac{ x_0(t)-x}{K}}^{-1}J(-z)|z|^q\,dz.
\end{equation}

Owing to \eqref{termeI2}, \eqref{ac-eq-I3} and \eqref{ac-eq-I4}, we get that, for some constant $C_2(q)>0$,
\begin{equation}\label{ac-eq-I2}
I_2\le w(t,x)\int_{\frac{ x_0(t)-x}{K}}^{+\infty}J(-z)\,dz+ C(q)w^{\beta}(t,x)\leq w(t,x)+C_2(q)w^{\beta}(t,x),
\end{equation}
since $\int _\R J=1$. Now, from \eqref{bidule}, \eqref{ac-eq-I1} and \eqref{ac-eq-I2},  we get,  for $t> 0$ and $x>  x_\gamma(t)$,
\begin{equation}
J*  m(t,x)- m(t,x)\le  \tilde C_1(q) w^{\frac{\alpha-1}{p}}(t,x) + C_2(q) w^{\beta}(t,x)
\end{equation}
where $\tilde C_1(q)=C_1K^{\alpha-1}(q)$. As a result,
$$
\partial_t  m -(J*   m- m)-f( m)\ge  w^{\beta}\left(\gamma-r_0-C_2(q)-\tilde C_1(q) w^{\frac{\alpha -1}{p}-\beta}\right).
$$

We are now close to conclusion. To validate the above computations we need $q=p(\beta-1)<1$, and in the above inequality we need the exponent $\frac{\alpha -1}{p}-\beta$ to be nonnegative. In view of \eqref{alpha-beta}, these two conditions reduce to $p\leq \frac{\alpha -1}{\beta}$, so that we make the optimal choice $p=\frac{\alpha -1}{\beta}$. For this choice of $p$, and thus of $q$, we now choose $\gamma \ge \gamma^*:=\max\{\gamma_0,\gamma_1(K),\gamma_2(A_0),r_0+C_2(q)+\tilde C_1(q)\}$, so that the right hand side of the above inequality is positive. This completes the construction of the supersolution. \end{proof}

Equipped with the above supersolution, we can now prove \eqref{upper}.

\begin{proof}[Proof of \eqref{upper}] In view of Assumption \ref{ass:u0} on the initial datum $u_0$, we can assume, up to a shift in space, that $u_0\leq v_0=m(0,\cdot)$ where $v_0$ is as \eqref{v0}. It therefore follows from the comparison principle that $u(t,x)\leq m(t,x)$, where $m(t,x)$ is the supersolution \eqref{def-m} with $p=\frac{\alpha -1}{\beta}$ and $\gamma\geq \gamma ^*$. Hence, for $\lambda \in(0,1)$, the super level set $\Gamma_\lambda(t)$ of $u$ is included in that of $m$. The latter can be explicitly computed, and we deduce
$$
x_\lambda (t)\leq  \left[ \left(\frac{1}{\lambda}\right)^{\beta -1} +\gamma  (\beta-1) t    \right]^{\frac{\beta}{(\alpha -1)(\beta -1)}}\leq \left[2 \gamma (\beta-1)t\right]^{\frac{\beta}{(\alpha -1)(\beta-1)}},
$$
for $t\geq T_\lambda$, with $T_\lambda >0$ large enough. This concludes the proof of \eqref{upper}. \end{proof}

To complete this subsection, it remains to prove Claim \ref{ac-cla-I4}.

\begin{proof}[Proof of  Claim \ref{ac-cla-I4}] Clearly $-\frac 1K\leq u\leq 0$ so that for $K$ large enough, say $K\geq K_0$, we have $(1+u)^q\ge 1+\frac{2}{q}u$. Hence
$$
 \frac{[1+u]^q-1}{1-\frac{\sigma}{x^q}}\ge \frac{2}{q}\frac{u}{1-\frac{\sigma}{x^q}}\ge \frac{2}{qK}\frac{ x_0(t)-x}{x^{1-q}}\frac{1}{x^q-\sigma},
$$
 since $\frac{ x_0(t)}{Kx}-\frac{1}{K}\leq u \leq -\frac{1}{x}$. Since $x>  x_0(t)$ we can write $x=\theta  x_0(t)$ for some $\theta \in (1,\infty)$.   Plugging this  in the right hand side of the above inequality, using $\sigma= x_0^q(t)-1$, and rearranging the terms we achieve  
\begin{align}
\frac{[1+u]^q-1}{1-\frac{\sigma}{x^q}}&\ge -\frac{2}{qK}\frac{\theta-1}{\theta-\theta^{1-q}+\theta^{1-q} x_0^{-q}(t)}\nonumber\\
&= -\frac{2}{qK}\frac{\theta-1}{\theta-\theta^{1-q}}\frac{\theta-\theta^{1-q}}{\theta-\theta^{1-q}+\theta^{1-q} x_0^{-q}(t)}\nonumber\\
&\geq -\frac{2}{qK}\frac{\theta-1}{\theta-\theta^{1-q}}\nonumber\\
&=-\frac{2}{qK}\frac{1-\frac{1}{\theta}}{1-\frac{1}{\theta^{q}}}\label{ac-eq-esti2-dl}.
\end{align}
Since there is $C(q)>0$ such that $\frac{1-\frac{1}{\theta}}{1-\frac{1}{\theta^{q}}}\leq C(q)$ for all $\theta >1$, we end up with
$$
\frac{[1+u]^q-1}{1-\frac{\sigma}{x^q}}\ge -\frac{2}{qK}C(q).
$$
The claim is then proved by taking  $K\ge K(q):= \max \left\{K_0, \frac{4C(q)}{q}\right\}$.
\end{proof}

\subsection{Lower bound on the speed of the super level sets}

Here we prove the lower bound \eqref{lower} of Theorem \ref{th::speedacc}. To measure the  acceleration  in this context, we use a subsolution that fills the space with a superlinear speed.  The construction of this subsolution is inspired by that   in \cite{Garnier2011} for nonlocal diffusion but in a KPP situation, and that in \cite{Alfaro2015a} in a Allee effect situation but for local diffusion.

\medskip

\noindent {\bf Step one.} It consits in using diffusion to gain an algebraic tail at time $t=1$.

 By Assumption \ref{ass:u0} on the initial datum $u_0$, we can construct a nonincreasing $\tilde u _0$ such that $\tilde u_0\leq u_0$ and 
\begin{equation}\label{c0}
\tilde u_0(x)=\begin{cases}
c_0 &\text{ for }  x \le -R_0-1\\
0 &\text{ for } x \ge -R_0,
\end{cases}
\end{equation}
for some $0<c_0<1$ and $R_0>0$. From the comparison principle, it is enough to prove \eqref{lower} for $u(t,x)$ the solution of \eqref{eq} starting from $\tilde u_0$.

Since $f$ is nonnegative, the comparison principle also implies $u(t,x)\ge v(t,x)$ for all $t>0$, $x\in \R$, where $v(t,x)$ is the solution of the linear problem 
\begin{align*}
&\partial_t v(t,x)= J* v(t,x)-v(t,x), \quad t>0,\, x\in \R,\\
&v(0,x)=\tilde u_0(x).
\end{align*}
Again from the comparison principle we get  $v(t,x)\ge e^{-t}(\tilde u_0(x)+t J* \tilde u_0(x))$, and thus   $v(1,x)\ge e^{-1}\left(\tilde u_0(x)+J* \tilde u_0(x)\right)$. In particular, for $x>0$ we have
$$
v(1,x)\geq e^{-1}J* \tilde u_0(x)\ge e^{-1}c_0\int_{R_0+1+x}^{\infty}J(z)\,dz\ge \frac{c_0/C}{\alpha -1}\frac{1}{(R_0+1+x)^{\alpha -1}},
$$ 
where we have used the tail estimate \eqref{hyp-queue-droite-noTW}. As a result, we can find a small enough $d>0$ such that
\begin{equation}\label{ac:def-v0}
u(1,x)\geq v(1,x) \ge v_0(x):= 
\begin{cases}
d &\text{ for }  x \le 1\\
\frac{d}{x^{\alpha-1}} &\text{ for } x \ge 1.
\end{cases}
\end{equation}
Hence, from the comparison principle and up to a shift in time, it is enough to 
prove \eqref{lower} for $u(t,x)$ the solution of \eqref{eq} starting from $v_0$, which we do below.

\medskip

\noindent {\bf Step two.} It consists in the construction of the subsolution.

Let us consider the  function $g(s):=s(1-Bs)$, with $B>\frac{1}{2d}$. We have  $g(s)\le 0$ for all $s\ge \frac{1}{B}$ and  $g(s)\le \frac{1}{4B}\le d$ for all $s\ge 0$.

As in the previous subsection,  let  $w(\cdot,x)$ denote the solution of the Cauchy problem
$$
\frac{dw}{dt}(t,x)=\gamma w^{\beta}(t,x), \quad w(0,x)=v_0(x),
$$
that is
 $$
 w(t,x)=\frac{1}{\left(v_0^{1-\beta}(x) -\gamma(\beta -1) t\right)^{\frac{1}{\beta -1}}}, 
 $$
where $v_0$ is defined in  \eqref{ac:def-v0}. Notice that $w(t,x)$ is not defined for all times. When $x\le 1$, $w(t,x)$ is defined for $t\in[0,\frac{1}{d^{\beta -1}\gamma(\beta-1)})$, whereas for $x>1$,   $w(t,x)$ is defined for $t\in\left[0,T(x):=\frac {x^{(\alpha-1)(\beta -1)}}{d^{\beta -1}\gamma(\beta-1)}\right)$.  Let us define
 \begin{equation}\label{ac-def-xB}
 x_{B}(t):=d^{\frac{1}{\alpha-1}} \left[B^{\beta-1}2^{ \beta -1}+\gamma (\beta -1)t \right]^{\frac{1}{(\alpha-1)(\beta-1)}}>1,
 \end{equation}
 so that $w(t,x_B(t))=\frac{1}{2B}$.

For  $x>1$ and $0< t< T(x)$, we compute 
 \begin{align*}
 &\partial_x w(t,x)= v_0'(x) v_0^{-\beta}(x)w^{\beta}(t,x)  \le 0,\\
 &\partial_{xx}w(t,x)=  v_0^{-\beta}(x)w^\beta(t,x)\left(v_0''(x)+  \beta \frac{( v_0'(x))^2}{v_0(x)}\left[ \left(\frac{w(t,x)}{v_0(x)}\right)^{\beta-1}-1\right]   \right)\geq 0.
 \end{align*}
Hence, for $t>0$,  $w(t,\cdot)$ is a decreasing convex function on  at least  $(x_B(t),\infty)$. 

Let us now define
 $$
 m(t,x):= 
 \begin{cases}
\frac{1}{4B} &\text{ for }  x \le x_B(t)\\
g(w(t,x)) &\text{ for } x > x_B(t).
\end{cases}
$$
Observe that: when $x>x_B(0)$, $m(0,x)=g(v_0(x))\leq v_0(x)$; when $x<1$, $m(0,x)=\frac{1}{4B}\leq d=v_0(x)$; when $1\leq x\leq x_B(0)$, $ v_0(x)\ge \frac{d}{x_B^{\alpha-1}(0)}=\frac{1}{2B}$ so that $m(0,x)\leq \frac{1}{4B}\leq v_0(x)$. Hence $m(0,x)\leq v_0(x)$ for all $x\in \R$. Let us now show that $m(t,x)$ a subsolution to \eqref{eq} for an appropriate choice of $\gamma$ and $B$. 
 
First, notice that, since $g\left(\frac{1}{2B}\right)=\frac{1}{4B}$ and $g'\left(\frac{1}{2B}\right)=0$, we see that $m\in C^{1}([0,\infty)\times \R)$. 
We compute
\begin{equation}\label{ac-eq-dtm}
\partial _t m(t,x)= 
 \begin{cases}
0  &\text{ for }  x \le x_B(t)\\
\gamma  w^{\beta}(t,x)\left(1 -2B w(t,x)\right)  &\text{ for } x > x_B(t).
\end{cases}
\end{equation}
Also, since $f$ satisfies \eqref{hyp-f-zero} and \eqref{hyp-f-un}, there exists a small $\delta >0$ such that $f(u)\geq \delta u^\beta (1-u)$ for all $0\leq u \leq 1$. As a result, we see that
\begin{equation}\label{ac-eq-f}
f(m(t,x))\geq 
 \begin{cases}
C_0w^\beta(t,x_B(t))  &\text{ for }  x \le x_B(t)\\
C_0w^{\beta}(t,x)  &\text{ for } x > x_B(t),
\end{cases}
\end{equation}
where  $C_0:=\frac{\delta}{2^\beta}\left(1- \frac{1}{4B}\right)$.  As far as the nonlocal diffusion term is concerned, thanks to the monotone behavior of  $m$, we have
$$
J* m(t,x) -m(t,x) \ge \int_{x}^{+\infty}J(x-y)(m(t,y)-m(t,x))\,dy=:\I(x),
$$ 
and we estimate $\I(x)$ below.

Assume first $x\le x_B(t)$, so that $m(t,x)=m(t,x_B(t))$ and by using the fundamental theorem of calculus
\begin{align*}
\I(x) &\geq   \int_{x_B(t)}^{+\infty}J(x-y)(m(t,y)-m(t,x_B(t)))\,dy, \\
&= \int_{x_B(t)}^{+\infty}\int_0^1J(x-y)(y-x_B(t))\partial_xm(t,x_B(t)+s(y-x_B(t)))\,dyds,\\
&=\int_{0}^{+\infty}\int_0^1J(x-x_B(t)-z)z \partial_xm(t,x_B(t)+sz) \,dzds.
\end{align*} 
  Now, $w(t,\cdot)$ being a positive decreasing convex function  in $(x_B(t),\infty)$, we have, for any $sz>0$,  
$$
\partial_xm(t,x_B(t)+sz)=\partial_xw(t,x_B(t)+sz)\left(1 -2Bw(t,x_B(t)+sz)\right)\ge \partial_xw(t,x_B(t)),
$$
so that
$$
J* m(t,x) -m(t,x)\ge \partial_xw(t,x_B(t)) \int_{0}^{+\infty}J(x-x_B(t)-z)z \,dz \ge \partial_xw(t,x_B(t))\int_{\R}J(z)|z| \,dz.
$$
As a result 
 \begin{equation}\label{ac-eq-Jm1}
 J* m(t,x) -m(t,x)\ge C v_0'(x_B(t)) v_0^{-\beta}(x_B(t)) w^{\beta}(t,x_B(t)),\quad \forall x\leq x_B(t), 
 \end{equation}
where $C:=\int_{\R}J(z)|z| \,dz$.

Similarly, when $x> x_B(t)$  by using the fundamental theorem of calculus, we get
\begin{align*}
\I(x)&=\int_{0}^{+\infty}\int_0^1J(-z)z  \partial_x w(t,x+sz)\left(1 -2Bw(t,x+sz)\right) \,dzds.
\end{align*} 
Then, by using the convexity and the monotonicity of $w(t,\cdot)$ in $(x_B(t),\infty)$,  we achieve 
\begin{equation}\label{ac-eq-Jm2}
J* m(t,x) -m(t,x)\ge C\partial_x w(t,x)=Cv_0'(x)v_0^{-\beta}(x)w^{\beta}(t,x), \quad \forall x>x_B(t).
\end{equation}

Collecting \eqref{ac-eq-dtm}, \eqref{ac-eq-f}, \eqref{ac-eq-Jm1} and \eqref{ac-eq-Jm2}, we end up with
$$
(\partial _t m -(J*m-m)-f(m))(t,x) \leq 
 \begin{cases}
-w^{\beta}(t,x_B(t))\left[C_0+h(t,x)\right] &\text{ for }  x \le x_B(t)\\
 -w^{\beta}(t,x)\left[C_0+h(t,x) -\gamma\right] &\text{ for } x > x_B(t),
\end{cases}
$$
where  
$$
h(t,x)= 
 \begin{cases}
Cv_0'(x_B(t))v_0^{-\beta}(x_B(t)) &\text{ for }  x \le x_B(t)\\
Cv_0'(x)v_0^{-\beta}(x)  &\text{ for } x > x_B(t).
\end{cases}
$$

We now choose $\gamma\le \frac{C_0}{2}$. In view of the above inequalities, to complete the construction of the subsolution $m(t,x)$, it suffices to find a condition on $B$ so that $h(t,x)\geq -\frac{C_0}{2}$ for all $t>0$, $x\in \R$. From the definition of $h(t,x)$ and that of $v_0(x)$ in \eqref{ac:def-v0}, this corresponds to achieve
$$
x^{(\beta-1)(\alpha - 1)-1}\le  \frac{C_0d^{\beta-1}}{2C(\alpha -1)}, \quad \text{ for all } t>0, x\geq x_B(t).
$$
Since $(\beta-1)(\alpha -1)<1$, this reduces to the following condition on $x_B(0)$
\begin{equation*} 
x_B(0)\ge \left(\frac{C_0d^{\beta-1}}{2C(\alpha -1)}\right)^{\frac{1}{1-(\beta-1)(\alpha - 1)}}. \label{ac-x0t-ineq}  
\end{equation*}
From \eqref{ac-def-xB} we have $x_B(0)=(2Bd)^{\frac{1}{\alpha -1}}$. Hence, in view of the definition of $C_0$, the above inequality holds by selecting $B\geq B_0$, with $B_0>0$ large enough. This concludes the construction of the subsolution $m(t,x)$.

\medskip
\noindent {\bf Step three.} It consits in using the subsolution to prove the lower estimate in \eqref{lower}.

Fix $\gamma>0$ and $B_0 >0$ as in the previous step so that $m(t,x)$ is a subsolution. From the comparison principle we get $m(t,x)\le u(t,x)$, for all $t>0$ and $x \in \R$. Recall that $m(t,x_{B_0}(t))=\frac{1}{4B_0}$ and that $u(t,\cdot)$ is nonincreasing (since initial datum $v_0$ is) so that
\begin{equation}
\label{u-grand}
u(t,x)\geq \frac{1}{4B_0}, \quad \forall x\leq x_{B_0}(t).
\end{equation}
In particular, for any $0<\lambda \leq \frac{1}{4B_0}$, the \lq\lq largest'' element $x_\lambda(t)$ of the super level set $\Gamma _\lambda (t)$ has to satisfy
$$
x_\lambda(t)\geq x_{B_0}(t)\geq d^{\frac{1}{\alpha -1}}[\gamma (\beta -1)t]^{\frac{1}{(\alpha-1)(\beta-1)}},
$$
which provides the lower estimate in \eqref{lower}.

It now remains to obtain a similar bound for a given $\frac{1}{4B_0}<\lambda <1$. Let us denote by $w(t,x)$ the solution of \eqref{eq} starting from a nonincreasing $w_0$ such that
 \begin{equation}\label{initial-data-w}
w_0(x)= \begin{cases} \frac{1}{4B_0} &\text{
if } x\leq -1
\\
0 &\text{ if  }x\geq 0.
\end{cases}
\end{equation}
It follows from  Proposition \ref{prop:level-set} $(ii)$  that there is a time $\tau_{\lambda}>0$  such that
\begin{equation}
\label{w-grand}
 w(\tau_{\lambda},x)>\lambda, \quad \forall x\leq 0.
 \end{equation}
 On the other hand, it follows from \eqref{u-grand} and  the definition \eqref{initial-data-w} that
 $$
 u(T,x)\ge w_0(x-x_{B_0}(T)),\quad \forall T\geq 0, \forall x\in \R,
 $$
 so that the comparison principle yields
 $$
 u(T+\tau,x)\geq w(\tau,x-x_{B_0}(T)),\quad \forall T\geq 0, \forall \tau \geq 0, \forall x\in \R.
 $$
  In view of \eqref{w-grand}, this implies that
 $$
 u(T+\tau_{\lambda},x)>\lambda,\quad \forall T\geq 0, \forall x\leq x_{B_0} (T).
 $$
 Hence, for $t\geq \tau _\lambda$,   the above implies  
 $$
 x_\lambda(t)\geq x_{B_0}(t-\tau _\lambda)=d^{\frac{1}{\alpha-1}} \left[B_0^{\beta-1}2^{ \beta -1}+\gamma (\beta -1)(t-\tau_{\lambda}) \right]^{\frac{1}{(\alpha-1)(\beta-1)}} \geq \underline C t^{\frac{1}{(\alpha-1)(\beta-1)}},
 $$
provided $t\geq T_\lambda '$, with $T_\lambda '>\tau _\lambda$ large enough. This concludes the proof of the lower estimate in \eqref{lower}. \qed

\medskip

\noindent \textbf{Acknowledgement.} 
J. Coville acknowledges support from the \lq\lq ANR JCJC'' project MODEVOL: ANR-13-JS01-0009 and  the ANR \lq\lq DEFI'' project NONLOCAL: ANR-14-CE25-0013.  

\bibliographystyle{siam}    
\bibliography{monostable}

\end{document}